\documentclass[11pt]{amsart}
\theoremstyle{plain}
\newtheorem{thm}{Theorem}[section]
\newtheorem{theorem}[thm]{Theorem}

\newtheorem{lemma}[thm]{Lemma}

\newtheorem{proposition}[thm]{Proposition}
\theoremstyle{definition}
\newtheorem{remark}[thm]{Remark}
\newtheorem{remarks}[thm]{Remarks}

\numberwithin{equation}{section}

 \title [(Relative) dynamical degrees in the algebraic setting]{(Relative) dynamical degrees of rational maps over an algebraic closed field}
 \author{Tuyen Trung Truong}
    \address{School of mathematics, Korea Institute for Advanced Study, Seoul 130-722, Republic of Korea}
 \email{truong@kias.re.kr}
\thanks{}
    \date{\today}
    \keywords{Algebraic closed field, Chow's moving lemma, Cone in a projective space, Dynamical degrees, Rational maps, Relative dynamical degrees, Resolution of singularities}
    \subjclass[2010]{37F, 14D, 32U40, 32H50}

\begin{document}
\maketitle
\begin{abstract}
The main purpose of this paper is to define dynamical degrees for rational maps over an algebraic closed field of characteristic zero and prove some basic properties (such as log-concavity) and give some applications. We also define relative dynamical degrees and prove a "product formula" for dynamical degrees of semi-conjugate rational maps in the algebraic setting. The main tools are the Chow's moving lemma and a formula for the degree of the cone over a subvariety of $\mathbb{P}^N$. The proofs of these results are valid as long as resolution of singularities are available (or more generally if appropriate birational models of the maps under consideration are available). This observation is applied for the cases of surfaces and threefolds over a field of positive characteristic. 
\end{abstract}

\section{Introduction}
 
One important tool in Complex Dynamics is dynamical degrees for dominant meromorphic selfmaps. They are bimeromorphic invariants of a meromorphic selfmap $f:X\rightarrow X$ of a compact K\"ahler manifold $X$. The $p$-th dynamical degree $\lambda _p(f)$  is the exponential growth rate of the spectral radii of the pullbacks $(f^n)^*$ on the Dolbeault  cohomology group $H^{p,p}(X)$. For a surjective holomorphic map $f$, the dynamical degree $\lambda _p(f)$ is simply the spectral radius of $f^*:H^{p,p}(X)\rightarrow H^{p,p}(X)$.  Fundamental results of Gromov \cite{gromov} and Yomdin \cite{yomdin} expressed the topological entropy of a surjective holomorphic map in terms of its dynamical degrees: $h_{top}(f)=\log \max _{0\leq p\leq {\dim (X)}}\lambda _{p}(f)$. Since then, dynamical degrees have played a more and  more important role in dynamics of meromorphic maps. In many results and conjectures in Complex Dynamics in higher dimensions, dynamical degrees play a central role. 

Let $X$ be a compact K\"ahler manifold of dimension $k$ with a K\"ahler form $\omega _X$, and let $f:X\rightarrow X$ be a dominant meromorphic map. For $0\leq p\leq k$, the $p$-th dynamical degree $\lambda _p(f)$ of $f$ is defined as follows 
\begin{equation} 
\lambda _p(f)=\lim _{n\rightarrow\infty}(\int _X(f^n)^*(\omega _X^p)\wedge \omega _X^{k-p})^{1/n}=\lim _{n\rightarrow\infty}r_p(f^n)^{1/n},
\label{Equation01}\end{equation}
where $r_p(f^n)$ is the spectral radius of the linear map $(f^n)^*:H^{p,p}(X)\rightarrow H^{p,p}(X)$ (See Russakovskii-Shiffman \cite{russakovskii-shiffman} for the case where $X=\mathbb{P}^k$, and Dinh-Sibony \cite{dinh-sibony10}\cite{dinh-sibony1} for the general case. The existence of the limit in the right hand side  of Equation (\ref{Equation01}) is non-trivial and has been proved using regularization of positive closed currents. This limit is important in establishing a "product formula" for dynamical degrees of semi-conjugate maps, see below for more details. Earlier, dynamical degrees for rational maps of complex projective manifolds was defined using liminf on the right hand side of Equation (\ref{Equation01}), see Section 2 in \cite{guedj4} (see also remarks therein for previous attempts). The dynamical degrees are log-concave, in particular $\lambda _1(f)^2\geq \lambda _2(f)$. In the case $f^*:H^{2,2}(X)\rightarrow H^{2,2}(X)$ preserves the cone of psef classes (i.e. those $(2,2)$ cohomology classes which can be represented by positive closed $(2,2)$ currents), then we have an analog $r_1(f)^2\geq r_2(f)$ (see Theorem 1.2 in \cite{truong2}).

For meromorphic maps of compact K\"ahler manifolds with invariant fibrations, a more general notion called relative dynamical degrees has been defined by Dinh and Nguyen in \cite{dinh-nguyen}. (Here, by a fibration we simply mean a dominant rational map, without any additional requirements. Hence it is more general than the standard one which requires that the fibers are connected.) Via "product formulas" proved in \cite{dinh-nguyen} and \cite{dinh-nguyen-truong1}, these relative dynamical degrees provide a very useful tool to check whether a meromorphic map is primitive (i.e.  has no invariant fibrations over a base which is of smaller dimension and not a point).  Roughly speaking, primitive maps (first defined by D.-Q. Zhang) are those which do not come from smaller dimensional manifolds, hence are "building blocks" from which all meromorphic maps can be constructed. 

Recently, there have been works on birational maps of surfaces over an algebraic closed field of arbitrary characteristic. As some examples, we refer the readers  to \cite{esnault-srinivas, xie, blanc-cantat, esnault-oguiso-yu, oguiso}. In these works, (relative) dynamical degrees also play an important role. Because of this and  for further applications in algebraic dynamics, it is desirable to have a purely algebraic definition of (relative) dynamical degrees for a rational map over a field other than $\mathbb{C}$.  In this paper we give such a definition for an algebraic closed field of characteristic zero. We will also prove similar results for the cases of surfaces and threefolds over an algebraic closed field of positive characteristic (see the end of this introduction for more details).

Here is our first main result. We recall that if $X\subset \mathbb{P}^N_K$ is a smooth projective variety over $K$, then we can define for any subvariety $W\subset X$ of pure dimension its degree ${\deg (W)}$ as the degree of $W$ viewed as a subvariety of $\mathbb{P}^N_K$.   

\begin{theorem}
Let $K$ be an algebraic closed field of characteristic zero, $X$ a smooth projective variety over $K$, $f:X\rightarrow X$ a dominant meromorphic map. Let $H_X$ be an ample divisor on $X$. Then 

1. For any $0\leq p\leq \dim (X)$, the following limit exists 
\begin{eqnarray*}
\lambda _p(f)=\lim _{n\rightarrow\infty}({\deg}((f^n)^*(H_X^p)))^{1/n}.
\end{eqnarray*}

2. The dynamical degrees $\lambda _p(f)$ are birational invariants.

3. The dynamical degrees are log-concave, that is $\lambda _p(f)^2\geq \lambda _{p-1}(f)\lambda _{p+1}(f)$, for $1\leq p\leq \dim (X)$. Moreover, $\lambda _p(f)\geq 1$ for all $0\leq p\leq \dim (X)-1$, and $\lambda _0(f)=1$.

4. Let $N^p(X)$ be the group of algebraic cycles of codimension $p$ on $X$ modulo numerical equivalence (see the next section for more details), and $N^p_{\mathbb{R}}(X)=N^p(X)\otimes _{\mathbb{Z}}\mathbb{R}$. Let $||.||$ be any norm on $N^p_{\mathbb{R}}(X)$, and $||(f^n)^*_p||$ the corresponding norm of the linear map $(f^n)^*:N^p_{\mathbb{R}}(X)\rightarrow N^p_{\mathbb{R}}(X)$. Then
\begin{eqnarray*}
\lambda _p(f)=\lim _{n\rightarrow\infty}||(f^n)^*_p||^{1/n}.
\end{eqnarray*}
\label{TheoremDynamicalDegreeAlgebraicCase}\end{theorem}     

As an application of this theorem, we will prove a result on the simplicity of the first dynamical degree of a rational map, similar to the main result in \cite{truong2}.

The idea for the proof of the theorem is as follows. To prove 1), we estimate the degree of the strict transform by $f$ of a codimension $p$ variety $W$ in terns of the degrees of $W$ and $f^*(H^p)$. This is achieved using Chow's moving lemma and the fact that for $W\subset \mathbb{P}^N$ the degree of the cone $C_L(W)$, over $W$ with vertex a linear subspace $L\subset \mathbb{P}^N$, is the same as that of $W$. For the proof of 2) we estimate the degree of the strict transform by a composition $f\circ g$ of a codimension $p$ variety $W$. For the proof of 3) we apply the Grothendieck-Hodge  index theorem.  For the proof of 4), we use a special norm $||.||_1$, to be defined in the next section.  

. \begin{remarks}We now give some remarks on Theorem \ref{TheoremDynamicalDegreeAlgebraicCase}. 

1) In the case of compact K\"ahler manifolds, the dynamical degrees are first defined using the Dolbeault cohomology groups $H^{p,p}(X)$. The proof of Theorem \ref{TheoremDynamicalDegreeAlgebraicCase} in that setting is given using properties of positive closed currents. If $X$ is a complex projective manifold, then dynamical degrees can be  also defined by using the groups $N^p(X)$. However, to show that the dynamical degrees defined using  the groups $H^{p,p}(X)$ and $N^p(X)$ are the same, an important property of complex projective manifolds were used. That is, a K\"ahler form (in particular, a smooth closed $(1,1)$ form representing an ample divisor) dominates any smooth closed $(1,1)$ form. Similarly, for any norm $||.||$ on $H^{p,p}(X)$, the norms of the classes of the form $\alpha ^p$, where $\alpha$ are ample divisors, dominate the norms of other cohomology classes. For a smooth projective variety defined on an arbitrary algebraic closed field of characteristic zero, the analogue of $H^{p,p}(X)$ is the algebraic de Rham groups $H^p(X,\Omega _X^p)$. Since we lack the notions of positivity for the classes in $H^p(X,\Omega _X^p)$, it is not clear that the property mentioned above of the norms on $H^{p,p}(X) $ carries out to the algebraic setting. If, however, we have such a property for norms on $H^p(X,\Omega _X^p)$, then dynamical degrees can be defined using the groups $H^p(X,\Omega _X^p)$. A similar observation applies for other cohomology groups, for example $l$-adic cohomology groups. We note that for {\bf automorphisms of surfaces} over an algebraic closed field of arbitrary characteristic, the expected equality between the dynamical degrees defined on the Neron-Severi group and the $l$-adic cohomology groups has been proved by Esnault and Srinivas \cite{esnault-srinivas}.   
    
2) The constants $C$ appearing in the estimates for the proof of Theorem \ref{TheoremDynamicalDegreeAlgebraicCase} are given explicitly in terms of the dimension of $X$ and the degree of $X$ in a given embedding $\iota :X\subset \mathbb{P}_K^N$.  
\end{remarks}

Our next main result shows that relative dynamical degrees can also be defined for semi-conjugate rational maps over an algebraic closed field of characteristic zero.
\begin{theorem}
Let $X$ and $Y$ be smooth rational varieties, of corresponding dimensions $k$ and $l$, over an algebraic closed field of characteristic zero. Let $f:X\rightarrow X$, $g:Y\rightarrow Y$ and $\pi :X\rightarrow Y$ be dominant rational maps such that $\pi \circ f=g\circ \pi$. Let $H_X$ be an ample class on $X$ and $H_Y$ an ample class on $Y$. Then,  for any $0\leq p\leq k-l$, there is a birational invariant number $\lambda _p(f|\pi )$, called the $p$-th relative dynamical degree, Moreover,

1) If $\pi$ is a regular map, the following limit exists
\begin{eqnarray*}
\lim _{n\rightarrow\infty}\deg ((f^n)^*(H_X^p).\pi ^*(H_Y^l))^{1/n},
\end{eqnarray*}
and equals $\lambda _p(f|\pi )$. 

2) We have $\lambda _p(f|\pi )\geq 1$ for all $0\leq p\leq k-l$, and $\lambda _0(f|\pi )=1$.

3) The relative dynamical degrees are log-concave, that is $\lambda _{p}(f|\pi )^2\geq \lambda _{p-1}(f|\pi ).\lambda _{p+1}(f|\pi )$.
\label{TheoremRelativeDynamicalDegrees}\end{theorem}

Using these dynamical degrees, we can prove the following "product formula" for dynamical degrees of semi-conjugate maps. 
\begin{theorem}
Let assumptions be as in Theorem \ref{TheoremRelativeDynamicalDegrees}. For any $0\leq p\leq k$ we have
\begin{eqnarray*}
\lambda _p(f)=\max _{0\leq j\leq l,~0\leq p-j\leq k-l}\lambda _j(g)\lambda _{p-j}(f|\pi ).
\end{eqnarray*}
\label{TheoremProductFormula}\end{theorem}

\begin{remark} 
It follows from the proofs of our results here that only resolutions of singularities are needed (or more generally if appropriate birational models of the maps under consideration are available). Thus, for surfaces and threefolds over fields of positive characteristic, the above results also hold. We will discuss more on this in the last Section. 
\end{remark}

{\bf Acknowledgements.} We would like to thank Mattias Jonsson for suggesting the question of defining dynamical degrees on fields other than the complex field. We would like to thank Claire Voisin, Charles Favre, Tien-Cuong Dinh and H\'el\`ene Esnault for their help. We thank Keiji Oguiso for his interest in the results of the paper. 

\section{Preliminaries on algebraic cycles}

Throughout the section, we fix an algebraic closed field $K$ of {\bf arbitrary} characteristic. Recall that a projective manifold over $K$ is a non-singular subvariety of a projective space $\mathbb{P}_K^N$. We will recall the definition and some results on algebraic cycles, the Chow's moving lemma and the Grothendieck-Hodge index theorem. We then arrive at a useful result on the intersection of two cycles and define a norm $||.||_1$ which will be used in the proof of Theorem  \ref{TheoremDynamicalDegreeAlgebraicCase}.

\subsection{Algebraic cycles}
Let $X\subset \mathbb{P}_K^N$ be a projective manifold of dimension $k$ over an algebraic closed field $K$ of arbitrary characteristic. A $q$-cycle on $X$ is a finite sum $\sum n_i[V_i]$, where $V_i$ are $q$-dimensional irreducible subvarieties of $X$ and $n_i$ are integers. The group of $q$-cycles on $X$, denoted $Z_q(X)$, is the free abelian group on the $p$-dimensional subvarieties of $X$ (see Section 1.3 in Fulton \cite{fulton}). A $q$-cycle $\alpha$ is effective if it has the form
\begin{eqnarray*}
\alpha =\sum _i n_i[V_i],
\end{eqnarray*}
where $V_i$ are irreducible subvarieties of $X$ and $n_i\geq 0$.

Let $X$ and $Y$ be projective manifolds, and let $f:X\rightarrow Y$ be a morphism. For any irreducible subvariety $V$ of $X$, we define the pushforward $f_*[V]$ as follows. Let $W=f(V)$. If $\dim (W)<\dim (V)$, then $f_*[V]=0$. Otherwise, $f_*[V]=\deg (V/W)[W]$. This gives a pushforward map $f_*:Z_q(X)\rightarrow Z_q(Y)$ (see Section 1.4 in \cite{fulton}). 

We refer the readers to \cite{fulton} for the definitions of rational and algebraic equivalences of algebraic cycles. Roughly speaking, two algebraic cycles are rationally equivalent if they are elements of a family of algebraic cycles parametrized by $\mathbb{P}^1$. Similarly, two algebraic cycles are algebraically equivalent if they are elements of  a family of algebraic cycles parametrized by a smooth algebraic variety. The groups of  $q$-cycles modulo rational and algebraic equivalences are denoted by $A_q(X)$ and $B_q(X)$.  

We write $Z^p(X)$, $A^p(X)$ and $B^p(X)$ for the corresponding groups of cycles of codimension $p$. Since $X$ is smooth, we have an intersection product $A^p(X)\times A^q(X)\rightarrow A^{p+q}(X)$, making $A^*(X)$ a ring, called the Chow's ring of $X$ (see Sections 8.1 and 8.3 in \cite{fulton}).  

For a dimension $0$ cycle $\gamma =\sum _im_i[p_i]$ on $X$, we define its degree to be $\deg (\gamma )=\sum _im_i$. We say that a cycle $\alpha \in A^{p}(X)$ is numerically equivalent to zero if and only $\deg (\alpha .\beta )=0$ for all $\beta \in A^{k-p}(X)$ (see Section 19.1 in \cite{fulton}). The group of codimension $p$ algebraic cycles modulo numerical equivalence is denoted by $N^p(X)$. These are finitely generated free abelian groups (see Example 19.1.4 in \cite{fulton}). The first group $N^1(X)$ is a quotient of the Neron-Severi group $NS(X)=B^1(X)$. The latter is also finitely generated, as proved by Severi and Neron. We will use the vector spaces $N^p_{\mathbb{R}}(X)=N^p(X)\otimes _{\mathbb{Z}}\mathbb{R}$ and $N^p_{\mathbb{C}}(X)=N^p(X)\otimes _{\mathbb{Z}}\mathbb{C}$ in defining dynamical degrees and in proving analogs of  Theorems 1.1 and 1.2 in \cite{truong2}.

\begin{remark}
We have the following inclusions: rational equivalence $\subset $ algebraic equivalence $\subset$ numerical equivalence.
\end{remark}

\subsection{Chow's moving lemma}

Let $X$ be a projective manifold of dimension $k$ over $K$. If $V$ and $W$ are two irreducible subvarieties of $X$, then either $V\cap W=\emptyset$ or any irreducible component of $V\cap W$ has dimension at least $\dim (V)+\dim (W)-k$. We say that $V$ and $W$ are properly intersected if any component of $V\cap W$ has dimension exactly $\dim (V)+\dim  (W)-k$. When $V$ and $W$ intersect properly, the intersection $V.W$ is well-defined as an effective $\dim (V)+\dim (W)-k$ cycle. 

Given $\alpha =\sum _{i}m_i[V_i]\in Z_q(X)$ and $\beta =\sum _{j}n_j[W_j]\in Z_{q'}(X)$, we say that $\alpha .\beta$ is well-defined if every component of $V_i\cap W_j$ has the correct dimension. Chow's moving lemma says that we can always find $\alpha '$ which is rationally equivalent to $\alpha$ so that $\alpha '.\beta $ is well-defined. Since in the sequel we will need to use some specific properties of such cycles $\alpha '$, we recall here a construction of such cycles $\alpha '$, following the paper Roberts \cite{roberts}. See also the paper Friedlander-Lawson \cite{friedlander-lawson} for a generalization to moving families of cycles of bounded degrees. 

Fixed an embedding $X\subset \mathbb{P}^N_K$, we choose a linear subspace $L\subset \mathbb{P}^N_K$ of dimension $N-k-1$ such that $L\cap X=\emptyset$. For any irreducible subvariety $Z$ of $X$ we denote by $C_L(Z)$ the cone over $Z$ with vertex $L$ (see Example 6.17 in the  book Harris \cite{harris}). For any such $Z$, $C_L(Z).X$ is well-defined and has the same dimension as $Z$, and moreover $C_L(Z).X-Z$ is effective (see Lemma 2 in \cite{roberts}).

Let $Y_1,Y_2,\ldots , Y_m$ and $Z$ be irreducible subvarieties of $X$. We define the excess $e(Z)$ of $Z$ relative to $Y_1,\ldots ,Y_m$ to be the maximum of the integers $$\dim (W)+k-\dim (Z)-\dim (Y_i),$$ 
where $i$ runs from $1$ to $m$, and $W$ runs through all components of $Z\cap Y_i$, provided that one of these integers is non-negative. Otherwise, the excess is defined to be $0$. 

More generally, if $Z=\sum _im_i[Z_i]$ is a cycle, where $Z_i$ are irreducible subvarieties of $X$, we define $e(Z)=\max _ie(Z_i)$. We then also define the cone $C_L(Z)=\sum _im_iC_{L}(Z_i)$. 
         
The main lemma (page $93$) in \cite{roberts} says that for any cycle $Z$ and any irreducible subvarieties $Y_1,\ldots ,Y_m$, then $e(C_L(Z).X-Z)\leq \max (e(Z)-1,0)$ for generic linear subspace $L\subset \mathbb{P}^N$ of dimension $N-k-1$ such that $L\cap X=\emptyset$.

Now we can finish the proof of Chow's moving lemma as follows (see Theorem page 94 in \cite{roberts}). Given $Y_1,\ldots ,Y_m$ and $Z$ irreducible varieties on $X$. If $e=e(Z)=0$ then $Z$ intersect properly $Y_1,\ldots ,Y_m$, hence we are done. Otherwise, $e\geq 1$. Applying the main lemma, we can find linear subspaces $L_1,\ldots ,L_e\subset \mathbb{P}^N_K$ of dimension $N-k-1$, such that if $Z_0=Z$ and $Z_i=C_{L_i}(Z_{i-1}).X-Z_{i-1}$ for $i=1,\ldots ,e=e(Z)$, then $e(Z_i)\leq e-i$. In particular, $e(Z_e)=0$. It is easy to see that
\begin{eqnarray*}
Z=Z_0=(-1)^eZ_e+\sum _{i=1}^{e}(-1)^{i-1}C_{L_i}(Z_{i-1}).X.    
\end{eqnarray*} 

It is known that there are points $g\in Aut(\mathbb{P}_K^N)$ such that $(gC_{L_i}(Z_{i-1})).X$ and $(gC_{L_i}(Z_{i-1})).Y_j$ are well-defined for $i=1,\ldots ,e$ and $j=1,\ldots ,m$. We can choose a rational curve in $Aut(\mathbb{P}_K^N)$ joining the identity map $1$ and $g$, thus see that $Z$ is rationally equivalent to
\begin{eqnarray*}
Z'=(-1)^eZ_e+\sum _{i=1}^{e}(-1)^{i-1}(gC_{L_i}(Z_{i-1})).X. 
\end{eqnarray*}
By construction, $e(Z')=0$, as desired. 

\subsection{Grothendieck-Hodge index theorem}

Let $X\subset \mathbb{P}^N_K$ be a projective manifold of dimension $k$. Let $H\subset \mathbb{P}^N_K$ be a hyperplane, and let $\omega _X =H|_X$. We recall that $N^p(X)$, the group of codimension $p$ cycles modulo the numerical equivalence, is a finitely generated free abelian group. We define $N^p_{\mathbb{R}}(X)=N^p(X)\otimes _{\mathbb{Z}}\mathbb{R}$ and $N^p_{\mathbb{C}}(X)=N^p(X)\otimes _{\mathbb{Z}}\mathbb{C}$. These are real (and complex) vector spaces of real (and complex) dimension equal $rank(N^p(X))$. For $p=1$, it is known that $\dim  _{\mathbb{R}} (N^1_{\mathbb{R}}(X)) =rank(NS(X))=:\rho$, the rank of the Neron-Severi group of $X$ (see Example 19.3.1 in \cite{fulton}). 

We define for $u,v\in N^1_{\mathbb{C}}(X)$ the Hermitian form 
\begin{eqnarray*}
\mathcal{H}(u,v)=\deg (u.\overline{v}.\omega _X^{k-2}).
\end{eqnarray*}
Here the degree of a complex $0$-cycle $\alpha +i\beta$ is defined to be the complex number $\deg (\alpha )+i\deg (\beta )$. The analogue of Hodge index theorem for compact K\"ahler manifolds is the Grothendieck-Hodge index theorem (see \cite{grothendieck}), which says that $\mathcal{H}$ has signature $(1,\rho -1)$. 

\subsection{Some norms on the vector spaces $N^p_{\mathbb{R}}(X)$ and $N^p_{\mathbb{C}}(X)$}

Given $\iota :X\subset \mathbb{P}^N_K$ a projective manifold of dimension $k$, let $H\in A^1(\mathbb{P}^N)$ be a hyperplane and $\omega _X=H|_{X}=\iota ^*(H)\in A^1(X)$. For an irreducible subvariety $V\subset X$ of codimension $p$, we define the degree of $V$ to be $\deg (V)=$ the degree of the dimension $0$ cycle $V.\omega _X^{k-p}$, or equivalently $\deg (V)=$degree of the variety $\iota _*(V)\subset \mathbb{P}^N$. Similarly, we define for an effective codimension $p$ cycle $V=\sum _{i}m_i[V_i]$ (here $m_i\geq 0$ and $V_i$ are irreducible), the degree $\deg (V)=\sum _im_i\deg (V_i)$. This degree is extended to vectors in $N^p_{\mathbb{R}}(X)$. Note that the degree map is a numerical equivalent invariant.   

As a consequence of Chow's moving lemma, we have the following result on intersection of cycles

\begin{lemma}
Let $V$ and $W$ be irreducible subvarieties in $X$. Then the intersection $V.W\in A^*(X)$ can  be represented as $V.W=\alpha _1-\alpha _2$, where $\alpha _1,\alpha _2\in A^*(X)$ are effective cycles and $\deg (\alpha _1),\deg (\alpha _2)\leq C\deg (V)\deg (W)$, where $C>0$ is a constant independent of $V$ and $W$. 
\label{LemmaDegreeOfIntersections}\end{lemma}
\begin{proof}
Using Chow's moving lemma, $W$ is rationally equivalent to 
\begin{eqnarray*}
W'=\sum _{i=1}^e(-1)^{i-1}gC_{L_i}(W_{i-1}).X+(-1)^eW_e,
\end{eqnarray*}
where $W_0=W$, $W_i=C_{L_i}(W_{i-1}).X-W_{i-1}$, $C_{L_i}(W_{i-1})\subset \mathbb{P}^N_K$ is a cone over $W_{i-1}$, and $g\in Aut(\mathbb{P}^N_K)$ is an automorphism. Moreover, $gC_{L_i}(W_{i-1}).X$, $gC_{L_i}(W_{i-1}).V$ and $W_e.V$ are all well-defined. We note that $e\leq k=\dim (X)$, and for any $i=1,\ldots ,e$ 
\begin{eqnarray*} 
\deg (W_i)&\leq& \deg (gC_{L_i}(W_{i-1}).X)\leq \deg (gC_{L_i}(W_{i-1}))\deg (X)\\
&=&\deg (C_{L_i}(W_{i-1})).\deg (X)=\deg (W_{i-1})\deg (X).
\end{eqnarray*}
Here we used that $\deg (C_{L_i}(W_{i-1})=\deg (W_{i-1})$ (see Example 18.17 in \cite{harris}), and $\deg (gC_{L_i}(W_{i-1})=\deg (C_{L_i}(W_{i-1})$ because $g$ is an automorphism of $\mathbb{P}^N$ (hence a linear map).  

Therefore, the degrees of $W_i$ are all $\leq (\deg (X))^k\deg (W)$. By definition, the intersection product $V.W\in A^*(X)$ is given by $V.W'$, which is well-defined. We now estimate the degrees of each effective cycle $gC_{L_i}(W_{i-1})|_X.V$ and $W_e.V$. Firstly, we have by the projection formula
\begin{eqnarray*}
\deg (gC_{L_i}(W_{i-1})|_X.V)&=&\deg (\iota _*(gC_{L_i}(W_{i-1})|_X.V))=\deg (gC_{L_i}(W_{i-1}).\iota _*(V))\\
&=&\deg (C_{L_i}(W_{i-1})).\deg (V)\leq \deg (X)^k\deg (W)\deg (V).
\end{eqnarray*}  
Finally, we estimate the degree of $W_e.V$. Since $W_e.V$ is well-defined, we can choose a linear subspace $L\subset \mathbb{P}^N$ so that $C_L(W_e).X$ and $C_L(W_e).V$ are well-defined. Recall that $C_L(W_e)-W_e$ is effective, we have
\begin{eqnarray*}
\deg (V.W_e)\leq \deg (V.C_L(W_e)|_{X})=\deg  (V).\deg (C_L(W_e))\leq \deg (X)^k\deg (V)\deg (W). 
\end{eqnarray*}

From these estimates, we see that we can write
\begin{eqnarray*}
V.W'=\alpha _1-\alpha _2,
\end{eqnarray*}
where $\alpha _1,\alpha _2$ are effective cycles and $\deg (\alpha _1),\deg (\alpha _2)\leq C\deg (V)\deg (W)$, where $C=k.\deg (X)^k$ is independent of $V$ and $W$.
\end{proof}

Using this degree map, we define for an arbitrary vector $v\in N^p_{\mathbb{R}}(X)$, the norm 
\begin{equation}
\|v\| _1=\inf \{\deg (v_1)+\deg (v_2):~v=v_1-v_2,~v_1, v_2\in N^p_{\mathbb{R}}(X) \mbox{ are effective}\}.
\label{LabelDefinitionNorm}\end{equation}
That this is actually a norm follows easily from Lemma \ref{LemmaDegreeOfIntersections} and that the bilinear form $N^p(X)\times N^{k-p}(X)\rightarrow \mathbb{Z}$, $(v,w)\mapsto \deg (v.w)$ is non-degenerate. If $v\in N^p_{\mathbb{R}}(X)$ is effective, then $\|v\|_1=\deg (v)$. Since $N^p_{\mathbb{R}}(X)$ is of finite dimensional, any norm on it is equivalent to $\|\cdot \|_1$. We can also complexify these norms to define norms on $N^p_{\mathbb{C}}(X)$.

\section{Dynamical degrees}

In the first subsection we consider pullback and strict transforms of algebraic cycles by rational maps. In the second subsection we define dynamical degrees and prove some of their basic properties. In the last subsection we define $p$-stability. We assume throughout this section that the field $K$ is of characteristic zero.

\subsection{Pullback and strict transforms of algebraic cycles by rational maps}
Let $X$ and $Y$ be two projective manifolds and $f:X\rightarrow Y$ a dominant rational map. Then we can define the pushforward operators $f_*:A_q(X)\rightarrow A_q(Y)$ and pullback operators $f^*:A^p(Y)\rightarrow A^p(X)$ (see Chapter 16 in \cite{fulton}). For example, there are two methods to define the pullback operators: 

Method 1: Let $\pi _X,\pi _Y:X\times Y\rightarrow X,Y$ be the two projections, and let $\Gamma _f$ be the graph of $f$. For $\alpha \in A^p(Y)$, we define $f^*(\alpha )\in A^p(X)$ by the following formula
\begin{eqnarray*}
f^*(\alpha )=(\pi _X)_*(\Gamma _f.\pi _Y^*(\alpha )). 
\end{eqnarray*}

Method 2: Let $\Gamma \rightarrow \Gamma _f$ be a resolution of singularities of $\Gamma _f$, and let $p,g:\Gamma \rightarrow X,Y$ be the induced morphisms. then we define 
\begin{eqnarray*}
f^*(\alpha )=p_*(g^*(\alpha )). 
\end{eqnarray*}

For the convenience of the readers, we recall here the arguments to show the equivalences of these two methods. Firstly, we show that the definition in Method 2 is independent of the choice of the resolution of singularities of $\Gamma _f$. In fact, let $\Gamma _1,\Gamma _2\rightarrow \Gamma _f$ be two resolutions of $\Gamma _f$ with the induced morphisms $p_1,g_1$ and $p_2,g_2$. Then there is another resolution of singularities $\Gamma \rightarrow \Gamma _f$ which dominates both $\Gamma _1$ and $\Gamma _2$ (e.g. $\Gamma$ is a resolution of singularities of the graph of the induced birational map $\Gamma _1\rightarrow \Gamma _2$). Let $\tau _1,\tau _2:\Gamma \rightarrow \Gamma _1,\Gamma _2$ the corresponding morphisms, and $p=p_1\circ \tau _1=p_2\circ \tau _2:\Gamma \rightarrow X$ and $g=g_1\circ \tau _1=g_2\circ \tau _2:\Gamma \rightarrow Y$ the induced morphisms. For $\alpha \in A^p(Y)$, we will show that $(p_1)_*(g_1^*\alpha )=p_*(g^*\alpha )=(p_2)_*(g_2^*\alpha )$. We show for example the equality $(p_1)_*(g_1^*\alpha )=p_*(g^*\alpha )$. In fact, we have by the projection formula 
\begin{eqnarray*}
p_*g^*(\alpha )&=&(p_1\circ \tau _1)_*(g_1\circ \tau _1)^*\alpha\\
&=&(p_1)_*(\tau _1)_*(\tau _1)^*(g_1)^*(\alpha )\\
&=&(p_1)_*(g_1)^*(\alpha ),
\end{eqnarray*}
as wanted. Finally, we show that the definitions in Method 1 and Method 2 are the same. By the embedded resolution of singularities (see e.g. the book \cite{kollar}), there is a finite blowup $\pi :\widetilde{X\times Y}\rightarrow X\times Y$ so that the strict transform $\Gamma$ of $\Gamma _f$ is smooth. Hence $\Gamma$ is a resolution of singularities of $\Gamma _f$, and $p=\pi _X\circ \pi \circ \iota , g=\pi _Y\circ \pi \circ\iota :\Gamma \rightarrow X,Y$ are the induced maps, where $\iota :\Gamma \subset \widetilde{X\times Y}$ is the inclusion map. For $\alpha \in A^p(Y)$, we have by the projection formula
\begin{eqnarray*}
p_*g^*(\alpha )&=&(\pi _X)_*\pi _* \iota _*\iota ^*\pi ^*\pi _Y^*(\alpha )=(\pi _X)_*\pi _* [\pi ^*\pi _Y^*(\alpha ).\Gamma ]\\
&=&(\pi _X)_*[\pi _Y^*(\alpha ).\pi _*(\Gamma )]=(\pi _X)_*[\pi _Y^*(\alpha ).\Gamma _f],
\end{eqnarray*}   
as claimed. 

In defining dynamical degrees and proving some of their basic properties, we need to estimate the degrees of the pullback and of strict transforms by a meromorphic map of a cycle. We present these estimates in the remaining of this subsection. We fix a resolution of singularities $\Gamma $ of the graph $\Gamma _f$, and let $p,g:\Gamma \rightarrow X,Y$ be the induced morphisms. By the theorem on the dimension of fibers (see e.g. the corollary of Theorem 7 in Section 6.3 Chapter 1 in the book Shafarevich \cite{shafarevich}), the sets
\begin{eqnarray*} 
V_{l}=\{y\in Y:~\dim (g^{-1}(y))\geq l\}
\end{eqnarray*}
are algebraic varieties of $Y$. We denote by $\mathcal{C}_g=\cup _{l>\dim (X)-\dim (Y)}V_l$ the critical image of $g$. We have the first result considering the pullback of a subvariety of $Y$

\begin{lemma}
Let $W$ be an irreducible subvariety of $Y$. If $W$ intersects properly any irreducible component of $V_l$ (for any $l>\dim (X)-\dim (Y)$), then $g^*[W]=[g^{-1}(W)]$ is well-defined as a subvariety of $\Gamma$. Moreover this variety represents the pullback $g^*(W)$ in $A^*(\Gamma )$.  
\label{LemmaGoodPullbackVariety}\end{lemma}
\begin{proof} (See also Example 11.4.8 in \cite{fulton}.) By the intersection theory (see Section 8.2 in \cite{fulton} and Theorem 3.4 in \cite{friedlander-lawson}), it suffices to show that $g^{-1}(W)$ has the correct dimension $\dim (X)-\dim (Y)+\dim (W)$. First, if $y\in W-\mathcal{C}_g$ then $\dim (g^{-1}(y))=\dim (X)-\dim (Y)$ by definition of $\mathcal{C}_g$. Hence $\dim (g^{-1}(W-\mathcal{C}_g)=\dim (W)+\dim (X)-\dim (Y)$. It remains to show that $g^{-1}(W\cap \mathcal{C}_g)$ has dimension $\leq \dim (X)+\dim (W)-\dim (Y)-1$. Let $Z$ be an irreducible component of $W\cap \mathcal{C}_g$. We define $l=\inf \{\dim (g^{-1}(y)):~y\in Z\}$. Then $l>\dim (X)-\dim (Y)$ and for generic $y\in Z$ we have $\dim (g^{-1}(y))=l$ (see Theorem 7 in Section 6.3 in Chapter 1 in \cite{shafarevich}). Let $V\subset V_l$ be an irreducible component containing $Z$.  By assumption $V.W$ has dimension $\dim (V)+\dim (W)-\dim (Y)$, hence $\dim (Z)\leq \dim (V)+\dim (W)-\dim (Y)$. We obtain $$\dim (g^{-1}(Z-V_{l+1}))=l+\dim (Z)\leq l+\dim (V)+\dim (W)-\dim (Y).$$ Since $g$ is surjective (because $f$ is dominant) and $V\not= Y$, it follows that $$\dim (X)-1\geq \dim (g^{-1}(V))\geq \dim (V)+l.$$ From these last two estimates we obtain 
\begin{eqnarray*}
\dim (g^{-1}(Z-V_{l+1}))&=&l+\dim (V)+\dim (W)-\dim (Y)\\
&\leq& \dim (X)-1+\dim (W)-\dim (Y).
\end{eqnarray*}
Since there are only a finite number of such components, it follows that $\dim (g^{-1}(W\cap \mathcal{C}_g))\leq \dim (W)+\dim (X)-\dim (Y)-1$, as claimed. 
\end{proof} 
 
We next estimate the degree of the pullback of a cycle. Fix an embedding $Y\subset \mathbb{P}^N_K$, and let $\iota :Y\subset \mathbb{P}^N_K$ the inclusion. Let $H\subset \mathbb{P}^N_K$ be a generic hyperplane and let $\omega _Y=H|_{Y}$. 
 
\begin{lemma} 
a) Let $p=0,\ldots ,\dim (Y)$, and let $Z\subset X$ be a proper subvariety. Then there is a linear subspace $H^p\subset \mathbb{P}^N_K$ of codimension $p$ such that $H^p$ intersects $Y$ properly, $f^*(\iota ^*(H^p))$ is well-defined as a subvariety of $X$, and $f^*(\iota ^*(H^p))$ has no component on $Z$. In particular, for any non-negative integer $p$, the pullback $f^*(\omega _Y^p)\in A^{p}(X)$ is effective. 

b) Let $W$ be an irreducible of codimension $p$ in $Y$. Then in $A^p(X)$, we can represent $f^*(W)$ by $\beta _1-\beta _2$, where $\beta _1$ and $\beta _2$ are effective and $\beta _1,\beta _2\leq C\deg (W)f^*(\omega _Y^p)$ for some constant $C>0$ independent of the variety $W$, the manifold $X$ and the map $f$.
\label{LemmaDegreeOfPullback}\end{lemma} 
\begin{proof}

Since by definition $f^*(W)=p_*g^*(W)$ and since $p_*$ preserves effective classes, it suffices to prove the lemma for the morphism $g$. We let the varieties $V_l$ as those defined before Lemma \ref{LemmaGoodPullbackVariety}.  

a) Let $H^p\subset \mathbb{P}^N_K$ be a generic codimension $p$ linear subspace. Then in $A^p(Y)$, $\omega _Y^p$ is represented by $\iota ^*(H^p)$. We can choose such an $H^p$ so that $H^p$ intersects properly $Y$, $g(Z)$ and all irreducible components of $V_l$ and $g(Z)\cap V_l$ for all $l>\dim (X)-\dim (Y)$. By Lemma \ref{LemmaGoodPullbackVariety}, the pullback $g^*(\iota ^*(H^p))=g^{-1}(\iota ^*(H^p))$ is well-defined as a subvariety of $\Gamma$. Moreover, the dimension of $g^{-1}(\iota ^*(H^p))\cap Z$ is less than the dimension of $g^{-1}(\iota ^*(H^p))$. In particular, $g^*(\iota ^*(H^p))$ is effective and has no component on $Z$. 

b) By Chow's moving lemma, $W$ is rationally equivalent to $\iota^*(\alpha _1)-\iota^*(\alpha _2)\pm W_e$, where $\alpha _1,\alpha _2\subset \mathbb{P}^N_K$  and $W_e\subset Y$ are subvarieties of codimension $p$, and they intersect properly $Y$ and all irreducible components of $V_l$ for all $l>\dim (X)-\dim (Y)$. Moreover, $\deg (\alpha _1),\deg (\alpha _2),\deg (W_e)\leq C\deg (W)$, for some $C>0$ independent of $W$. By the proof of Chow's moving lemma, we can find a codimension $p$ variety $\alpha\subset \mathbb{P}^N_K$ so that $\alpha$ intersect properly with $Y$ and all $V_l$, $\iota ^*(\alpha ) -W_e$ is effective, and $\deg (\alpha )\leq C\deg (W_e)$. Note that in $A^p(Y)$ we have $\iota ^*(\alpha _1)\sim \deg (\alpha _1)\omega _Y^p$, $\iota ^*(\alpha _2)\sim \deg (\alpha _2)\omega _Y^p$ and $\iota ^*(\alpha )\sim \deg (\alpha )\omega _Y^p$. Note also that $0\leq g^*(W_e)\leq g^*(\iota ^*(\alpha ))$. Therefore, in $A^p(\Gamma )$
\begin{eqnarray*}
g^*(W)\sim \deg (\alpha _1)g^*(\omega _Y^p)-\deg (\alpha _2)g^*(\omega _Y^p)\pm g^*(W_e),    
\end{eqnarray*}
where each of the three terms on the RHS is effective and $\leq C\deg (W)g^*(\omega _Y ^p)$ for some $C>0$ independent of $W$, $X$ and $f$.
\end{proof}
\begin{lemma}
Let $f:X\rightarrow Y$ be a rational map. For any $p=0,\ldots ,\dim (Y)-1$, we have
$$
f^{*}(\omega _Y^{p+1})\leq f^*(\omega _Y^p).f^*(\omega _Y)
$$
in $A^{p+1}(X)$.
\label{LemmaBoundForPullbackKahlerForms}\end{lemma}  
\begin{proof}
Let $Z\subset X$ be a proper subvariety containing $p(g^{-1}(\mathcal{C}_g))$ so that $p:\Gamma -p^{-1}(Z)\rightarrow X-Z$ is an isomorphism. Then the restriction map $$p_0:\Gamma -g^{-1}(gp^{-1}(Z))\rightarrow X-p(g^{-1}(gp^{-1}(Z)))$$ is also an isomorphism, and the restriction map $$g_0:\Gamma -g^{-1}(gp^{-1}(Z))\rightarrow Y- gp^{-1}(Z)$$ has fibers of the correct dimension $\dim (X)-\dim (Y)$. 

Choose $H,H^p,H^{p+1}\subset \mathbb{P}^N_K$ be linear subspaces of codimension $1$, $p$ and $p+1$ such that $H^{p+1}=H\cap H^p$. We can find an automorphism $\tau \in \mathbb{P}^N_K$, so that $\tau (H),\tau (H^p), \tau (H^{p+1})$ intersects properly $Y$ and all irreducible components of $gp^{-1}(Z)$ and of  $V_l$ for all $l>\dim (X)-\dim (Y)$. For convenience, we write $H,H^p$ and $H^{p+1}$ for $\tau (H),\tau (H^p), \tau (H^{p+1})$, and $H_Y,H_Y^p$ and $H_Y^{p+1}$ for their intersection with $Y$. Then all the varieties $g^{-1}(H_Y),g^{-1}(H_Y^p)$ and $g^{-1}(H_Y^{p+1})$ have the correct dimensions, and have no components in $g^{-1}(gp^{-1}(Z))$. Hence the pullbacks $f^*(H_Y), f^*(H_Y^p)$ and $f^*(H_Y^{p+1})$ are well-defined as varieties in $X$ and has no components on $p(g^{-1}(gp^{-1}(Z)))$.

We next observe that the two varieties $f^*(H_Y)$ and $f^*(H_Y^p)$ intersect properly. Since $f^*(H_Y)$ is a hypersurface, it suffices to show that any component of $f^*(H_Y)\cap f^*(H_Y^p)$ has codimension $p+1$. Since $f^*(H_Y^p)$ has no component on $p(g^{-1}(gp^{-1}(Z)))$, the codimension of $f^*(H_Y)\cap f^*(H_Y^p)\cap p(g^{-1}(gp^{-1}(Z)))$ is at least $p+1$. It remains to show that $f^*(H_Y)\cap f^*(H_Y^p)\cap (X-p(g^{-1}(gp^{-1}(Z))))$ has codimension $p+1$. Since $p_0$ is an isomorphism, the codimension of the latter equals that of 
$$g^{-1}(H_Y)\cap g^{-1}(H_Y^p)\cap (\Gamma - g^{-1}(gp^{-1}(Z)))=g^{-1}(H_Y\cap H_Y^{p})\cap (\Gamma - g^{-1}(gp^{-1}(Z)))$$ 
which is $p+1$.

Therefore $f^*(H_Y).f^*(H_Y^{p+1})$ is well-defined as a variety of $X$, and on $X-p(g^{-1}(gp^{-1}(Z)))$ it equals
\begin{eqnarray*}
(p_0)*(g_0^*(H_Y).g_0^*(H_Y^p))=(p_0)_*(g_0^*(H_Y^{p+1}))=p_*g^*(H_Y^{p+1}).
\end{eqnarray*}
Since the latter has no component on $p(g^{-1}(gp^{-1}(Z)))$, it follows that $f^*(H_Y).f^*(H_Y^p)\geq f^*(H_Y^{p+1})$. From this inequality, we obtain the desired inequality in $A^{p+1}(X)$
\begin{eqnarray*}
f^*(\omega _Y^p).f^*(\omega _Y)\geq f^*(\omega _Y^{p+1}).
\end{eqnarray*} 
\end{proof}

Finally, we estimate the degree of a strict transform of a cycle. Define $$g_0=g|_{\Gamma -g^{-1}(\mathcal{C}_g)}: \Gamma -g^{-1}(\mathcal{C}_g)\rightarrow Y-\mathcal{C}_g.$$ Then $g_0$ is a proper morphism, and for any $y\in Y-\mathcal{C}_g$, $g_0^{-1}(y)$ has the correct dimension $\dim (X)-\dim (Y)$. Let $W\subset Y$ be a codimension $p$ subvariety. The inverse image $g_0^{-1}(W)=g^{-1}(W)\cap (\Gamma -g^{-1}(\mathcal{C}_g))\subset \Gamma -g^{-1}(\mathcal{C}_g)$ is a closed subvariety of codimension $p$ of $\Gamma -g^{-1}(\mathcal{C}_g)$, hence its closure $cl(g_0^{-1}(W))\subset \Gamma$ is a subvariety of codimension $p$, and we define $f^{o}(W)=p_*cl(g_0^{-1}(W))$. Note that a strict transform depends on the choice of a resolution of singularities $\Gamma$ of the graph $\Gamma _f$. (We can also define a strict transform more intrinsically using the graph $\Gamma _f$ directly, as in \cite{dinh-sibony3}.)

\begin{lemma}
Let $W\subset Y$ be a codimension $p$ subvariety. Then $f^o(W)$ is an effective cycle, and in $A^p(X)$
\begin{eqnarray*} 
f^{o}(W)\leq C\deg (W)f^*(\omega _Y^p),
\end{eqnarray*}
where $C>0$ is a constant independent of the the variety $W$, the manifold $X$  and the map $f$. 
\label{LemmaDegreeOfStrictTransform}\end{lemma}
\begin{proof}
That $f^0(W)$ is an effective cycle follows from the definition. It suffices to prove the lemma for the morphism $g:\Gamma \rightarrow Y$. By the proof of Chow's moving lemma, we can decompose $W$ as follows
\begin{eqnarray*}
W=\sum _{i=1}^e(-1)^{i-1}\iota ^*(C_{L_i}(W_{i-1}))+(-1)^eW_e,
\end{eqnarray*} 
where the variety $W_e$ intersects properly all irreducible components of $V_l$ for all $l>\dim (X)-\dim (Y)$, and $C_{i}(W_{i-1})\subset \mathbb{P}^N_K$ are subvarieties of codimension $p$ intersecting $Y$ properly (but may not intersect properly the irreducible components of $V_l$). Moreover, we have the following bound on the degrees
\begin{equation}
\deg (W_e),\deg (C_{L_i}(W_{i-1}))\leq C\deg (W),
\label{Equation0}\end{equation}
for all $i$, where $C>0$ is independent of $W$, $X$ and $f$. 

By the definition of $g^0$ we have
\begin{equation}
g^o(W)= \sum _{i=1}^e(-1)^{i-1}g^o(\iota ^*(C_{L_i}(W_{i-1})))+(-1)^eg^o(W_e).
\label{Equation1}\end{equation}

Note that $e\leq \dim (Y)$. We now estimate each term on the RHS of (\ref{Equation1}). Let $S\subset \mathbb{P}^N_K$ be a subvariety of codimension $p$ intersecting $Y$ properly (but may not intersecting properly the components of $V_l$). We first show that for any such $S$
\begin{equation}
g^o(\iota ^*(S))\leq \deg (S)g^*(\omega _Y^p),
\label{Equation2}\end{equation}
in $A^p(\Gamma )$. 

We can find a curve of automorphisms $\tau (t)\in Aut(\mathbb{P}^N_K)$ for $t\in \mathbb{P}^1_K$ such that for a dense Zariski open dense subset  $U\subset \mathbb{P}^1$, $\tau (t)S$ intersects properly $Y$ and all the irreducible components of $V_l$ (for $l>\dim (X)-\dim (Y)$)  for all $t\in U$. Let $\mathcal{S}\subset Y\times \mathbb{P}^1$ be the corresponding variety, hence for $t\in U\subset \mathbb{P}^1$, $\mathcal{S}_t=\iota ^*(\tau (t)S)\subset Y$. Since $S$ intersects $Y$ properly, we have $\mathcal{S}_0=\iota ^*(S)$. By the choice of $\mathcal{S}$, for any $t\in U$ the pullback $g^*(\mathcal{S}_t)$ is well-defined as a subvariety of $\Gamma$. 

We consider the induced map $G:\Gamma \times \mathbb{P}^1\rightarrow Y\times \mathbb{P}^1$ given by the formula 
\begin{eqnarray*}
G(z,t)=(g(z),t).
\end{eqnarray*}
We define by $G_0$ the restriction map $G_0:\Gamma \times U\rightarrow Y\times U$. By the choice of the variety $\mathcal{S}$, the inverse image $$G_0^{-1}(\mathcal{S})=G^{-1}(\mathcal{S})\cap  (\Gamma \times U)\subset \Gamma \times U$$ is a closed subvariety of codimension $p$, hence its closure $G^o(\mathcal{S})\subset \Gamma \times \mathbb{P}^1$ is a subvariety of codimension $p$. Moreover, for all $t\in U$ we have
\begin{eqnarray*}
G^o(\mathcal{S})_t=g^*(\mathcal{S}_t).
\end{eqnarray*}
Since the map $g_0:\Gamma -g^{-1}(\mathcal{C}_g)\rightarrow Y-\mathcal{C}_g$ has all fibers of the correct dimension $\dim (X)-\dim (Y)$, it follows that $$G^o(\mathcal{S})_0\cap (\Gamma -g^{-1}(\mathcal{C}_g))=g_0^{-1}(\iota ^*(S)).$$ 
In fact, let $G_1$ be the restriction of $G$ to $(\Gamma -\mathcal{C}_g)\times \mathbb{P}^1$. Then $$G_1^{-1}(\mathcal{S} )=G^{-1}(\mathcal{S})\cap [(\Gamma -\mathcal{C}_g)\times \mathbb{P}^1]\subset (\Gamma -\mathcal{C}_g)\times \mathbb{P}^1$$
is a closed subvariety of codimension $p$. Hence its closure, denoted by $\widetilde{G}^o(\mathcal{S})\subset \Gamma \times \mathbb{P}^1$ is a subvariety of codimension $p$. For $t\in U$, we have $\widetilde{G}^o(\mathcal{S})_t=g^*(\mathcal{S}_t)=G^o(\mathcal{S}_t)$, because on the one hand $\widetilde{G}^o(\mathcal{S})_t\subset G^{-1}(\mathcal{S})_t=g^*(\mathcal{S}_t)$, and on the other hand $g^*(\mathcal{S}_t)$ has no component on $g^{-1}(\mathcal{C}_g)$ and $\widetilde{G}^o(\mathcal{S})_t\cap (\Gamma -g^{-1}(\mathcal{C}_g))=g_0^{-1}(\mathcal{S}_t)$. Therefore $\widetilde{G}^o(\mathcal{S})=G^o(\mathcal{S})$ as varieties on $\Gamma \times \mathbb{P}^1$. In particular  
\begin{eqnarray*}
G^o(\mathcal{S})_0\cap (\Gamma -g^{-1}(\mathcal{C}_g))=\widetilde{G}^o(\mathcal{S})_0\cap (\Gamma -g^{-1}(\mathcal{C}_g))=g_0^{-1}(\iota ^*(S)),
\end{eqnarray*}
as claimed.

Hence 
\begin{eqnarray*}
g^o(\iota ^*(S))\leq G^o(\mathcal{S})_0
\end{eqnarray*}
as varieties on $\Gamma$. Since $G^o(\mathcal{S})_0$ is rationally equivalent to $G^o(\mathcal{S})_t$ for any $t$ in $U$, it follows that for all such $t$ we have
\begin{eqnarray*}
g^o(\iota ^*(S))\leq G^o(\mathcal{S})_t=g^*(\mathcal{S}_t)=\deg (S)g^*(\omega _Y^p),
\end{eqnarray*}
in $A^p(\Gamma )$. Hence (\ref{Equation2}) is proved. 

Now we continue the proof of the lemma. By (\ref{Equation2}) and the bound on degrees (\ref{Equation0}), for all $i=1,\ldots ,e$
\begin{eqnarray*}
g^o(\iota ^*(C_{L_i}(W_{i-1})))\leq C\deg (W)g^*(\omega _Y^p),
\end{eqnarray*}
in $A^p(\Gamma )$ where $C>0$ is independent of $W$, $X$ and $f$. 

It remains to estimate $g^o(W_e)$. By the choice of $W_e$, the pullback $g^*(W_e)$ is well-defined as a subvariety of $\Gamma$, hence by b) of Lemma \ref{LemmaDegreeOfPullback} and the bound on degrees (\ref{Equation0}) we have
\begin{eqnarray*} 
g^o(W_e)\leq g^*(W_e)\leq C\deg (W)g^*(\omega _Y^p),
\end{eqnarray*}
in $A^p(\Gamma )$, where $C>0$ is independent of $W$, $X$ and $f$. Thus the proof of the lemma is completed. 
\end{proof} 
  
\subsection{Dynamical degrees and some of their basic properties}
We define here dynamical degrees and prove some of their basic properties. When $K$ is the field of complex numbers, all of the results in this subsection were known.  Note that in this case (i.e. when $K=\mathbb{C}$), our approach here using Chow's moving lemma is different from the previous ones using "regularization of currents" (see \cite{russakovskii-shiffman} for the case $X=\mathbb{P}^N_{\mathbb{C}}$ the complex projective space and see \cite{dinh-sibony1}\cite{dinh-sibony10} for the case $X$ is a general compact K\"ahler manifold; see also \cite{guedj5}). Let $X$ be a projective manifold with a given embedding $\iota :X\subset \mathbb{P}^N_K$. We let $H\subset \mathbb{P}^N$ be a linear hyperplane, and let $\omega _X=H|_{X}$.

\begin{lemma}
Let $Y,Z$ be projective manifolds, and let $f:Y\rightarrow X,~g:Z\rightarrow Y$ be dominant rational maps. We fix an embedding $Y\subset \mathbb{P}^M_K$ and let $\omega _Y$ be the pullback to $Y$ of a generic hyperplane in $\mathbb{P}^M_K$. Then in $A^p(Z)$ 
\begin{eqnarray*} 
(f\circ g)^*(\omega _X^p) \leq C\deg (f^*(\omega _X^p))g^*(\omega _Y^p),
\end{eqnarray*}
where $C>0$ is independent of $f$ and $g$.
\label{LemmaDegreeOfCompositionMaps}\end{lemma}
\begin{proof}
We can find proper subvarieties $V_X\subset X,V_Y\subset Y$ and $V_Z\subset Z$ so that the maps $f_0:Y-V_Y\rightarrow X-V_X$ and $g_0:Z-V_Z\rightarrow Y-V_Y$ are regular and have all fibers of the correct dimensions. In fact, first consider the canonical projection $p_f:\Gamma _f\rightarrow X$, which is a surjective regular map between projective varieties. By the theorem on the dimension of fibers (see e.g. the corollary of Theorem 7 in Section 6.3 Chapter 1 in the book Shafarevich \cite{shafarevich}), there is a proper subvariety $V_1\subset X$ such that for any $V_X\subset X$ a proper subvariety such that $p_f:\Gamma _f-p_f^{-1}(V_X)\rightarrow X-V_X$ has all fibers of correct dimension. Now, $\Gamma _f$ is a modification of $Y$ via the canonical projection $\pi _f:\Gamma _f\rightarrow Y$, we can choose $V_X$ such that $\pi _f:\Gamma _f-p_f^{-1}(V_X)\rightarrow Y-\pi _f(p_f^{-1}(V_X))$ is an isomorphism. We can also arrange $V_X$ large enough such that the same argument applies to $V_Y=\pi _f(p_f^{-1}(V_X))$ will give that $Z- V_Z\rightarrow Y-V_Y$ is regular and has all fibers of the correct dimension. Then it follows that the composition $Z-V_Z\rightarrow X-V_X$ is regular and has all fibers of the correct dimension.   

Define by $(f\circ g)_0$ the restriction of $f\circ g$ to $Z-V_Z$. Then $(f\circ g)_0=f_0\circ g_0:Z-V_Z\rightarrow X-V_X$ and has all fibers of the correct dimension. We define the strict transforms $f^0$, $g^0$ and $(f\circ g)^0$ using these restriction maps $f_0,g_0$ and $(f\circ g)_0$. 

By Lemma \ref{LemmaDegreeOfPullback} a), we can find a linear subspace $H^p\subset \mathbb{P}^N_K$ so that $H^p$ intersects $X$ properly, $(f\circ g)^*(\iota ^*(H^p))$ is well-defined as a variety and has no component on $V_Z$, and $f^*(\iota ^*(H^p))$ is well-defined as a variety. Then
\begin{eqnarray*}
(f\circ g)^*(\iota ^*(H^p))=(f\circ g)^o(\iota ^*(H^p))
\end{eqnarray*} 
is the closure of $$(f\circ g)_0^{-1}(\iota ^*(H^p))=(f_0\circ g_0)^{-1}(\iota ^*(H^p))=(g_0)^{-1}f_0^{-1}(\iota ^*(H^p)).$$
Therefore $$(f\circ g)^*(\iota ^*(H^p))=g^of^o(\iota ^*(H^p))\leq g^of^*(\iota ^*(H^p)).$$
as subvarieties of $Z$. By Lemma \ref{LemmaDegreeOfStrictTransform}, we have the desired result. 
\end{proof}

Let $f:X\rightarrow X$ be a dominant rational map. Fix a number $p=0,\ldots ,k=\dim (X)$. Apply Lemma \ref{LemmaDegreeOfCompositionMaps} to $Y=Z=X$ and the maps $f^n,f^m$, we see that the sequence $n\mapsto \deg((f^n)^*(\omega _X^p))$ is sub-multiplicative. Therefore, we can define the $p$-th dynamical degree as follows
\begin{eqnarray*}
\lambda _p(f)=\lim _{n\rightarrow\infty}(\deg ((f^n)^*(\omega _X^p)))^{1/n}=\inf _{n\in \mathbb{N}}(C\deg ((f^n)^*(\omega _X^p)))^{1/n}.
\end{eqnarray*}
Here $C$ is the constant in Lemma \ref{LemmaDegreeOfCompositionMaps}. 

We now relate $\lambda _p(f)$ to the spectral radii $r_p(f^n)$ of the linear maps $(f^n)^*:N_{\mathbb{R}}^p(X)\rightarrow N_{\mathbb{R}}^p(X)$.

\begin{lemma}

a) There is a constant $C>0$ independent of $f$ so that
\begin{eqnarray*}
\|f^*(v)\|_1\leq C\|v\|_1\|f^*(\omega _X^p)\|_1,
\end{eqnarray*}
for all $v\in N_{\mathbb{R}}^p(X)$. Here the norm $\|\cdot \|_1$ is defined in (\ref{LabelDefinitionNorm}). Therefore if we denote by $f_p^*$ the linear map $f^*:N^p_{\mathbb{R}}(X)\rightarrow N^p_{\mathbb{R}}(X)$, and by $$\|A\|_1=\sup _{v\in N^p_{\mathbb{R}}(X), \|v\|_1=1}\|A(v)\|_1$$ the norm of a linear map $A:N^p_{\mathbb{R}}(X)\rightarrow N^p_{\mathbb{R}}(X)$ then
\begin{eqnarray*}
\frac{1}{\deg (\omega _X^p)}\|f^*(\omega _X^p)\|_1\leq \|f_p^*\|_1\leq C\|f^*(\omega _X^p)\|_1,
\end{eqnarray*}
here $C$ is the same constant as in the previous inequality.

b) There is a constant $C>0$ independent of $f$ so that $r_p(f)\leq C\|f^*(\omega _X^p)\|_1.$

c) We have $\lambda _p(f)=\lim _{n\rightarrow\infty}\|(f^n)_p^*\|_1^{1/n}\geq \limsup _{n\rightarrow\infty}(r_p(f^n))^{1/n}.$
\label{LemmaSpectralRadius}\end{lemma} 
\begin{proof}

a) Let $m=\dim _{\mathbb{Z}}N^p(X)$, and we choose varieties $v_1,\ldots ,v_m$ to be a basis for $N^p(X)$. Then $v_1,\ldots ,v_m$ is also a basis for $N_{\mathbb{R}}^p(X)$. We denote by $\|\cdot \|_2$ the max norm on $N_{\mathbb{R}}^p(X)$ with respect to the basis $v_1,\ldots ,v_m$, thus for $v=a_1v_1+\ldots a_mv_m$ 
\begin{eqnarray*} 
\|v\|_2=\max \{|a_1|,\ldots ,|a_m|\}.
\end{eqnarray*}
By Lemma \ref{LemmaDegreeOfPullback}, we can write each $f^*(v_j)$ as a difference $\alpha _j-\beta _j$ where $\alpha _j$ and $\beta _j$ are effective and $\deg (\alpha _j),\deg (\beta _j)\leq C\deg (v_j)\deg (f^*(\omega _X^p))$. Here $C>0$ is independent of the map $f$. In particular, $\|f^*v_j\|_1\leq C\deg (v_j)\deg (f^*(\omega _X^p))$ for any $j=1,\ldots ,m$. Therefore
\begin{eqnarray*}
\|f^*v\|_1&=&\|a_1f^*(v_1)+\ldots a_mf^*(v_m)\|_1\leq |a_1| \|f^*(v_1)\|_1+\ldots +|a_m|\|f^*(v_m)\|_1\\
&\leq&C\|v\|_2\|f^*(\omega _X^p)\|_1\leq C'\|v\|_1\|f^*(\omega _X^p)\|_1
\end{eqnarray*}
since any norm on $N_{\mathbb{R}}^p(X)$ is equivalent to $\|.\|_1$. The other inequalities follow easily from definition of $\|f_p^*\|_1$. Hence a) is proved.
 
b) Iterating a) we obtain
\begin{eqnarray*}
\|(f^*)^nv\|_1\leq C^{n-1}\|v\|_1\|f^*(\omega _X^p)\|_1^n,
\end{eqnarray*}
for all $n\in \mathbb{N}$ and $v\in N_{\mathbb{R}}^p(X)$. Taking supremum on all $v$ with $\|v\|_1=1$ and then taking the limit of the $n$-th roots when $n\rightarrow \infty$, we obtain
\begin{eqnarray*}
r_p(f)\leq C \|f^*(\omega _X^p)\|_1.
\end{eqnarray*}
Here $C>0$ is the same constant as in a).

c) The equality 
\begin{eqnarray*}
\lambda _p(f)=\lim _{n\rightarrow\infty}\|(f^n)_p^*\|_1^{1/n}
\end{eqnarray*}
follows from a) and the definition of the dynamical degree $\lambda _p(f)$. 

The inequality $$\lambda _p(f)\geq \limsup _{n\rightarrow\infty}(r_p(f^n))^{1/n}$$ follows from b) and the definition of the dynamical degree $\lambda _p(f)$.
\end{proof}

Using Lemma \ref{LemmaDegreeOfCompositionMaps}, it is standard (see e.g. \cite{dinh-sibony10}) to prove the following result
\begin{lemma} The dynamical degrees are birational invariants. More precisely, if $X,Y$ are projective manifolds of the same dimension $k$, $f:X\rightarrow X$ and $g:Y\rightarrow Y$ are dominant rational maps, and $\pi :X\rightarrow Y$ is a birational map so that $\pi \circ f=g\circ \pi$, then $\lambda _p(f)=\lambda _p(g)$ for all $p=0,\ldots ,k$. 
\label{LemmaDynamicalDegreeBirationalInvariant}\end{lemma}

Finally, using the Grothendieck-Hodge index theorem, we prove the log-concavity property of dynamical degrees.
\begin{lemma}
Let $f:X\rightarrow X$ be a rational map. For any $1\leq p\leq k-1=\dim (X)-1$
\begin{eqnarray*}
\lambda _p(f)^2\geq \lambda _{p-1}(f)\lambda _{p+1}(f).
\end{eqnarray*}
Moreover, $\lambda _p(f)\geq 1$ for all $0\leq p\leq k$.
\label{LemmaBoundForDynamicalDegrees}\end{lemma}  
\begin{proof}
To prove the log-concavity, it is enough to show that for any rational map $f$ and for any $1\leq p\leq k-1$ then
\begin{eqnarray*}
\deg (f^*\omega _X^p)^2\geq \deg (f^*\omega _X^{p-1}).\deg (f^*\omega _X^{p+1}).
\end{eqnarray*}
Let $\Gamma$ be a resolution of singularities for the graph of $f$ (which exists by Hironaka's theorem since we are working on characteristic zero), and let $\pi ,g:\Gamma \rightarrow X$ be the projections. Then, $f=g\circ \pi ^{-1}$. It follows that
\begin{eqnarray*}
\deg (f^*\omega _X^j)=g^*(\omega _X^{j}).\pi ^*(\omega _X^{k-j})
\end{eqnarray*}
for all $0\leq j\leq k$. Define $\alpha =g^*(\omega _X)$ and $\beta =\pi ^*(\omega _X)$. These are nef divisors on $\Gamma$, and the desired inequality is
\begin{eqnarray*}
(\alpha ^{p}.\beta ^{k-p})^2\geq (\alpha ^{p-1}.\beta ^{k-p+1})\times (\alpha ^{p+1}.\beta ^{k-p-1}).
\end{eqnarray*}
Let $\omega$ be an ample divisor on $\Gamma$, replacing $\alpha $ and $\beta$ by $n\alpha +\omega$ and $n\beta +\omega$ for $n$ large enough, it suffices to prove the above inequality for the case $\alpha$ and $\beta$ are both very ample divisors. Then $W=\alpha ^{p-1}.\beta ^{k-p-1}$ is an irreducible smooth surface, and what we want to prove becomes
\begin{eqnarray*}
(\alpha |_W.\beta |_W)^2\geq (\alpha |_W^2)\times (\beta |_W^2).
\end{eqnarray*} 
(Here, $\alpha |_W$ and $\beta |_W$ are the pullbacks of $\alpha$ and $\beta$ to $W$.) This inequality follows from the Grothendieck-Hodge index theorem for the surface $W$.

Finally, that $\lambda _p(f)\geq 1$ for all $0\leq p\leq k$ follows from the fact that $f^*(\omega _X^p) $ is an effective non-empty subvariety of $X$, hence $\deg (f^*(\omega _X^p))$ is a positive integer, and hence $\geq 1$. 

\end{proof}

\subsection{$p$-stability}
Let $X$ be a projective manifold, and $f:X\rightarrow X$ a dominant rational map. Given $p=0,\ldots ,k=\dim (X)$, we say that $f$ is $p$-stable if for any $n\in \mathbb{N}$, $(f^n)^*=(f^*)^n$ on $N^{p}_{\mathbb{R}}(X)$. Note that when $p=1$ then $1$-stable is the same as $(f^n)^*=(f^*)^n$ on hypersurfaces for any $n=1,2,\ldots $. In the case $K=\mathbb{C}$, this notion was first used in Fornaess - Sibony \cite{fornaess-sibony1}.

{\bf Example.} If $X$ is a surface and $f:X\rightarrow X$ is a birational map, then by results in \cite{diller-favre} $f$ has a birational model $f_1:X_1\rightarrow X_1$ which is $1$-stable.

\section{The simplicity of the first dynamical degree}   
\label{SubsectionApplications}

In this section we give some algebraic analogues of the results in \cite{truong2}, concerning the simplicity of the first dynamical degree of a rational map. Let $K$ be an algebraic closed field of characteristic zero. 

\begin{lemma} Let $X\subset \mathbb{P}^N_K$ be a projective manifold of dimension $k$. Let $\pi :Z\rightarrow X$ be a blowup of $X$ along a smooth  submanifold $W=\pi (E)$ of codimension exactly $2$. Let $E$ be the exceptional divisor and let $L$ be a general fiber of $\pi$. Let $\alpha$ be a vector in $N^1_{\mathbb{C}}(Z)$. 

i) In $A^*(X)$ we have
\begin{eqnarray*}
(\pi )_*(E.E)=-W.
\end{eqnarray*}

ii)  
\begin{eqnarray*}
\pi ^*(\pi )_*(\alpha )=\alpha +(\{\alpha \}.\{L\})E.
\end{eqnarray*}

iii)  
\begin{eqnarray*}
(\pi )_*(\alpha . E)=(\{\alpha \}.\{L\})W.
\end{eqnarray*}

iv) 
\begin{eqnarray*}
(\pi )_*(\alpha ).(\pi )_* (\overline{\alpha })-(\pi )_*(\alpha . \overline{\alpha })=|\{\alpha \}.\{L\}|^2W.
\end{eqnarray*}
\label{LemmaPullPushFormulaForOneBlowupAlgebraicCase}\end{lemma}
\begin{proof}
i) follows from the formula at the beginning of Section 4.3 in \cite{fulton}. Then ii), iii) and iv) follows from i) as in the proof of Lemma 2.1 in \cite{truong2}.
\end{proof}

\begin{proposition}
Let $X$ and $Y$ be projective manifolds, and $h:X\rightarrow Y$ a dominant rational map. Further, let $u\in N^1_{\mathbb{C}}(Y)$, then $h^*(u).h^*(\overline{u} )-h^*(u . \overline{u})\in N^2_{\mathbb{R}}(X)$ is effective.
\label{LemmaPullPushFormulaForMeromorphicMapsAlgebraicCase}\end{proposition}
\begin{proof}
The proof is identical with that of Proposition 2.2 in \cite{truong2}, the only difference here is that we use Hironaka's elimination of indeterminacies for rational maps on projective manifolds over algebraic closed fields of characteristic zero (see e.g. Corollary 1.76 in Koll\'ar \cite{kollar} and Theorem 7.21 in Harris \cite{harris}). (In the algebraic case, the Hironaka's elimination of indeterminacies for a rational map $f:X\rightarrow Y$ is a consequence of the basic monomialization theorem, applied to the ideal generated by the components of the map $f$ in an ambient projective space of $Y$.)
\end{proof}

Now we state the analogs of Theorems 1.1 and 1.2 in \cite{truong2}. We omit the proofs of these results here since they are similar to those of Theorems 1.1 and 1.2 in \cite{truong2}. 

\begin{theorem} Let $X\subset \mathbb{P}^N_K$ be a projective manifold of dimension $k$, and let $f:X\rightarrow X$ be a dominant rational map which is $1$-stable. Assume that $\lambda _1(f)^2>\lambda _2(f)$. Then $\lambda _1(f)$ is a simple eigenvalue of $f^*:N^1_{\mathbb{R}}(X)\rightarrow N^1_{\mathbb{R}}(X)$. Further, $\lambda _1(f)$ is the only eigenvalue of modulus greater than $\sqrt{\lambda _2(f)}$.  
\label{TheoremFirstDynamicalDegreeIsSimpleAlgebraicCase}\end{theorem}

We denote by $r_1(f)$ and $r_2(f)$ the spectral radii of the linear maps $f^*:N^1_{\mathbb{R}}(X)\rightarrow N^1_{\mathbb{R}}(X)$ and $f^*:N^2_{\mathbb{R}}(X)\rightarrow N^2_{\mathbb{R}}(X)$.
\begin{theorem}
Let $X\subset \mathbb{P}^N_K$ be a projective manifold, and let $f:X\rightarrow X$ be a dominant rational map. Assume that $f^*:N^2_{\mathbb{R}}(X)\rightarrow N^2_{\mathbb{R}}(X)$ preserves the cone of effective classes. Then  

1) We have $r_1(f)^2\geq r_2(f)$.

2) Assume moreover that $r_1(f)^2>r_2(f)$. Then $r_1(f)$ is a simple eigenvalue of $f^*:N^1_{\mathbb{R}}(X)\rightarrow N^1_{\mathbb{R}}(X)$. Further, $r_1(f)$ is the only eigenvalue of modulus greater than $\sqrt{r_2(f)}$.  

\label{TheoremTheCaseNotStableAlgebraicCase}\end{theorem}

\section{Relative dynamical degrees}

Let $X$ and $Y$ be smooth projective varieties over an algebraic closed field of characteristic zero, of dimensions $k$ and $l$ respectively. Let $f:X\rightarrow X$ and $g:Y\rightarrow Y$ be dominant rational maps. We say that $f$ and $g$ are semi-conjugate if there is a dominant rational map $\pi :X\rightarrow Y$ such that $\pi \circ f=g\circ \pi$. When the ground field is the field $\mathbb{C}$ of complex numbers, Dinh and Nguyen \cite{dinh-nguyen} defined relative dynamical degrees $\lambda _p(f|\pi )$. Their paper uses regularization of positive closed currents on compact K\"ahler manifolds. As mentioned above, for varieties over a field other than $\mathbb{C}$, we do not have such regularization results. In this section we give a purely algebraic definition of relative dynamical degrees for maps over an arbitrary algebraic closed field of characteristic zero. We also show that these relative dynamical degrees are birational invariants and are log-concave. 

First, we note that by Hironaka's resolution of singularities, we can make $\pi$ to be regular. We now first define relative dynamical degrees for the case where $\pi$ is regular. Then, by showing that these relative dynamical degrees are birational invariants, we can define the relative dynamical degrees for the general case where $\pi$ is merely rational. 

From now on, we assume that both $X$ and $Y$ are embedded in the same projective space $\mathbb{P}^N$. 

\subsection{The case $\pi$ is a regular map}

Because $Y$ has dimension $l$, a generic complete intersection of hyperplane sections $\omega _Y^l$ is a union of $\deg (Y)$ points. The preimages $\pi ^{-1}(\omega _Y^l)\subset X$ will play an important role in the definition of relative dynamical degrees. The advantage of the choice $\pi ^{-1}(\omega _Y^l)$ lies in that these are rationally equivalent and that $\pi ^{-1}(\omega _Y)$ is a nef class. 

\begin{lemma}
Let $\omega _{X,1}^p$ and $\omega _{X,2}^p$ be two generic complete intersections of hyperplane sections of $X$, and $\omega _{Y,1}^l$ and $\omega _{Y,2}^l$ two generic complete intersections of  hyperplane sections of $Y$. Then the intersections $f^*(\omega _{X,1}^p)\cap \pi ^{-1}(\omega _{Y,1}^l)$ and $f^*(\omega _{X,2}^p)\cap \pi ^{-1}(\omega _{Y,2}^l)$ are well-defined and have the same class in $N^*(X)$.
\label{LemmaRDD2}\end{lemma}
\begin{proof}
 By the proof of Lemma \ref{LemmaDegreeOfStrictTransform}, both $\pi ^{-1}(\omega _{Y,1}^l)$ and $\pi ^{-1}(\omega _{Y,2}^l)$ are fibers of the same pencil $\mathcal{S}$. Similarly, both $f^*(\omega _{X,1}^p) $ and $f^*(\omega _{X,2}^p)$ are fibers of the same pencil $\mathcal{T}$. Moreover, $\mathcal{T}\cap \mathcal{S}$ has the correct dimension, by dimensional consideration. From this, the conclusions of the lemma follow. \end{proof}

By Lemma \ref{LemmaRDD2}, the number 
\begin{eqnarray*}
\deg _p(f|\pi ):=\deg (f^*(\omega _{X}^p)\cap \pi ^{-1}(\omega _{Y}^l))
\end{eqnarray*}
is well-defined and is independent of the choices of generic complete intersections of hyperplane sections $\omega _X^p$ and $\omega _Y^l$. 

Our next objective is to show that the following limit exists
\begin{eqnarray*}
\lambda _p(f|\pi )=\lim _{n\rightarrow\infty}\deg _p(f^n|\pi )^{1/n}.
\end{eqnarray*}
These limits are the relative dynamical degrees we seek to define. To prove the existence of these limits, we will make use of another quantity associated to $f$. 

{\bf Notation.} Let $\mathcal{B}_f\subset X$ be a proper subvariety, called a "bad" set for $f$, outside it the map $f$ has good properties needed for our purpose. For example, $\mathcal{B}_f$ may contain the critical and indeterminate set of $f$, but it may also contains the critical and indeterminate sets of some iterates of $f$ and the pre images of some "bad" set $\mathcal{B}_g\subset Y$. We will not specify these "bad" sets in advance, but will determine them specifically for each purpose. 

Let $V$ be a subvariety of $X$, whose any irreducible component does not belong to the "bad" set $\mathcal{B}_f$.  Then similarly to the above, the strict pushforward $f_0(\omega _X^j\cap (V-\mathcal{B}_f))$ is well-defined for a generic complete intersection of  hyperplane sections $\omega _X^j$. Moreover, the class in $N^*(X)$ of the closure of $f_0(\omega _X^j\cap (V-\mathcal{B}_f))$ is independent of the choice of $\mathcal{B}_f$ and such a generic complete intersection of hyperplane sections $\omega _X^j$. We choose, in particular $V=\pi ^{-1}(\omega _Y^l)$ where $\omega _Y^l$ is a generic complete intersection of hyperplane sections of $Y$. Then the class in $N^*(X)$  of the closure of $f_0(\omega _X^j\cap (\pi ^{-1}(\omega _Y^l)-\mathcal{B}_f))$  is independent of the choice of such $\mathcal{B}_f$, $\omega _X^j$ and $\omega _Y^l$. For $j=k-l-p$, we denote the degree of the closure of $f_0(\omega _X^j\cap (\pi ^{-1}(\omega _Y^l)-\mathcal{B}_f))$ by $\deg '_p(f|\pi )$.
\begin{lemma}
For all $0\leq p\leq k-l$, we have: $\deg _p(f|\pi )=\deg '_p(f|\pi )$.
\label{LemmaRDD3}\end{lemma}
\begin{proof}
We choose a generic complete intersection of hyperplane sections $\omega _X^{k-l-p}$ of $X$. By definition
\begin{eqnarray*}
\deg _p(f|\pi )=\omega _X^{k-l-p}.Z,
\end{eqnarray*}
where $Z$ is the closure of $f^*(\omega _X^p).(\pi ^{-1}(\omega _Y^l)-\mathcal{B}_f)$. Since $\omega _X^{k-l-p}$ is a generic complete intersection of hyperplane sections, we have $\omega _X^{k-l-p}\cap Z'=\emptyset$, where $Z'=Z-f^*(\omega _X^p).(\pi ^{-1}(\omega _Y^l)-\mathcal{B}_f)$. Moreover, we can choose so that $\omega _X^{k-l-p}.f^*(\omega _X^p).(\pi ^{-1}(\omega _Y^l)-\mathcal{B}_f)$, which is a $0$-cycle, does not intersect any priori given subvariety of $X$. 

Let $\Gamma _f$ be the graph  of $f$, and let $\pi _1,\pi _2$ be the two projections $X\times X\rightarrow X$. We consider the following $0$-cycle on $X\times X$: $\pi _1^*(\omega _X^{k-l-p}.f^*(\omega _X^p).(\pi ^{-1}(\omega _Y^l)-\mathcal{B}_f)). \pi _2^*(\omega _X^p).\Gamma _f$. Denote by $\alpha$ this zero cycle. The degree of $(\pi _1)_*(\alpha )$ is $\deg _p(f|\pi )$ and the degree of $(\pi _2)_*(\alpha )$ is $\deg '_p(f|\pi )$. Then, the lemma follows from the fact that the degrees of the push forwards by $\pi _1$ and $\pi _2$ of a $0$-cycle on $X\times X$ are the same. 
 \end{proof}

Now we define relative dynamical degrees. We note that the next theorem is valid over a field of arbitrary characteristic, see the Remark inside the body of the proof. In particular, the Sard-Bertini's type theorem for characteristic zero is not needed in the algebraic case (compare with the proof of Proposition 3.3. in \cite{dinh-nguyen}).
\begin{theorem}
For any $0\leq p\leq k-l$, the following limit exists
\begin{eqnarray*}
\lambda _p(f|\pi ):=\lim _{n\rightarrow\infty}\deg '_p(f^n|\pi )^{1/n}=\lim _{n\rightarrow \infty}\deg _p(f^n|\pi )^{1/n}.
\end{eqnarray*}
We call $\lambda _p(f|\pi )$ the $p$-th relative dynamical degree of $f$ with respect to $\pi$.
\label{TheoremRDD4}\end{theorem}
\begin{proof}
By Lemma \ref{LemmaRDD3}, we only need to show that the first limit exists. To this end, it suffices to show that there is a constant $C>0$ independent of $m,n$ and $f$ such that
\begin{eqnarray*}
\deg '_p(f^{n+m}|\pi )\leq C\deg '_p(f^n|\pi ).\deg '_p(f^m|\pi ).
\end{eqnarray*}

We define $R=$ the closure of $f^n_0(\omega _X^{k-l-p}\wedge (\pi ^{-1}(\omega _Y^l)-\mathcal{B}))$, where $\mathcal{B}$ is a proper subvariety of $X$, depending on $f^n$, $f^m$ and $f^{n+m}$, so that the strict transform $f^m_0(R)$ is well-defined and equals the closure of $f^{n+m}_0(\omega _X^{k-l-p}\wedge (\pi ^{-1}(\omega _Y^l)-\mathcal{B}))$. By choosing $\omega _X^{k-l-p}$ and $\omega _Y^l$ generically, we can assume that $R$ is a subvariety of codimension $k-l-p$ of $\pi ^{-1}(g_*^n(\omega _Y^l))$ and has proper intersections with the "bad" set $\mathcal{B}$, using the property $\pi \circ f=g\circ \pi$. We note that by definition the degree of $R$ is $\deg '_p(f^n|\pi )$. 

We observe that in $N^*(X)$:
\begin{equation}
R\leq C\deg (R)\omega _X^{k-l-p}\wedge \pi ^{-1}(g_*^n(\omega _Y^l)),
\label{EquationRDD1}\end{equation}
here $C>0$ is a positive constant independent of $R$, $m,n$ and the maps $\pi$, $f$ and $g$. In fact, consider the embedding $\pi ^{-1}(g_*^n(\omega _Y^l))\subset \mathbb{P}^N$ induced from the embedding $X\subset \mathbb{P}^N$. By the Sard-Bertini's theorem on fields of characteristic zero,  $\pi ^{-1}(g_*^n(\omega _Y^l))$ is smooth (but may be reducible), we can apply the estimates from the usual Chow's moving lemma (see Section 2). Because the degree of $R$ is the same when considered as either a subvariety of $X$ or of $\pi ^{-1}(g_*^n(\omega _Y^l))$ (since we are embedding both of them into the same $\mathbb{P}^N$), we then obtain the desired inequality, with $C=l \deg (\pi ^{-1}(g_*^n(\omega _Y^l)))^l$. We note that $\omega _Y^l$ is a union of $\deg (Y)$ points, therefore $g_*^n(\omega _Y^l)$ is a union of $\deg (Y)$ points $y_1,\ldots ,y_{\deg (Y)}$. Each point $y_j$ belongs to a generic $\omega _{Y,j}^l$, and hence 
\begin{eqnarray*}
\deg (\pi ^{-1}(g_*^n(\omega _Y^l)))\leq \sum _{j=1}^{\deg (Y)}\deg (\pi ^{-1}(g_*^n(\omega _{Y,j}^l))=\deg (Y)C_1,
\end{eqnarray*}
where $C_1=\deg (\pi ^{-1}(g_*^n(\omega _{Y,j}^l))$ is independent of $j$. Hence the inequality (\ref{EquationRDD1}) is proved. 

{\bf Remark.} Here we give an alternative proof, which does not use the fact that $\pi ^{-1}(g_*^n(\omega _Y^l))$ is smooth, and hence is valid over a field of positive characteristic. Regardless whether $\pi ^{-1}(g_*^n(\omega _Y^l))$ is smooth or not, it has the correct dimension $k-l$. We then can use Lemma 2 in \cite{roberts}, to have that if $L\subset \mathbb{P}^N$ is any linear subspace of dimension $n-(k-l)-1$ and $C_L(R)$ is the cone over $R$ with vertex $L$, then $C_L(R)\cap \pi ^{-1}(g_*^n(\omega _Y^l))$ has the correct dimension $\dim (R)$. In particular, since $R$ is one component of $C_L(R)\cap \pi ^{-1}(g_*^n(\omega _Y^l))$ and the degree of $C_L(R)$ is bounded from above by $C\deg (R)$, we obtain Equation (\ref{EquationRDD1}). 

Using the same arguments as in the proof of Lemma \ref{LemmaDegreeOfStrictTransform}, we can show moreover that  in $N^*(X)$:
\begin{equation}
f^m_0(R)\leq C\deg (R)f^m_0(\omega _X^{k-l-p}\wedge \pi ^{-1}(g_*^n(\omega _Y^l))).
\label{EquationRDD2}\end{equation}
This completes the proof of Theorem \ref{TheoremRDD4}, by noting that $\pi ^{-1}(g_*^n(\omega _Y^l))$ is contained in at most $\deg (Y)$ fibers of the form $\pi ^{-1}(\omega _Y^l)$.
\end{proof}

Now we proceed to proving the log-concavity of relative dynamical degrees. We note that in the case the underlying field is $\mathbb{C}$, our proof here is different from and simpler than that used in Proposition 3.6 of \cite{dinh-nguyen}. 

\begin{lemma} The relative dynamical degrees are log-concave, that is $\lambda _{p+1}(f|\pi )\lambda _{p-1}(f|\pi )\leq \lambda _p(f|\pi )^2$ for all $1\leq p\leq k-l-1$. Moreover, for all $0\leq p\leq k-l$ we have $\lambda _p(f|\pi )\geq 1$. 
\label{LemmaLogConcavityRelativeDynamicalDegrees}\end{lemma}
\begin{proof}
The fact that $\lambda _p(f|\pi )\geq 1$ follows easily from that the numbers $\deg _p(f^n|\pi )$ are the degrees of certain subvarieties of $X$, hence must be positive integers. 

Now we prove the log-concavity of relative dynamical degrees. To this end, it suffices to show that 
\begin{eqnarray*}
\deg _{p+1}(f|\pi )\deg _{p-1}(f|\pi )\leq \deg _p(f|\pi )^2. 
\end{eqnarray*}
Let $\Gamma$ be a resolution of singularities of the graph of $f$, and let $\pi _1,\pi _2:\Gamma \rightarrow X$ be the two projections. Then, by definition, we have
\begin{eqnarray*}
\deg _p(f|\pi )=f^*(\omega _X^p).\pi ^*(\omega _Y^l).\omega _X^{k-l-p}=\pi _2^*(\omega _X)^p.(\pi \circ \pi _1)^*(\omega _Y)^l.\pi _1^*(\omega _X)^{k-l-p}.
\end{eqnarray*}
This is an intersection of nef classes on $\Gamma$. Therefore, we can approximate them by $ \mathbb{Q}$-ample classes, and use the usual Grothendieck-Hodge index theorem to prove the log-concavity, as in the proof of  Lemma \ref{LemmaBoundForDynamicalDegrees}.
\end{proof}

\subsection{The general case: $\pi$ is a rational map} Here we use the results in the previous subsection to define relative dynamical degrees in the general case where the map $\pi :X\rightarrow Y$ is merely a rational map. To this end, we proceed to showing that relative dynamical degrees are birational invariants. 

Let $f:X\rightarrow X$, $g:Y\rightarrow Y$ and $\pi :X\rightarrow Y$ be dominant rational maps such that $\pi \circ f=g\circ \pi$. That is, $f$ and $g$ are semi-conjugate maps. Applying Hironaka's theorem, we can make a birational model $f_1:X_1\rightarrow X_1$ of $f$ such that the corresponding map $\pi _1:X_1\rightarrow Y$ is regular. Then, as in the previous subsection we can define relative dynamical degrees $\lambda _p(f_1|\pi _1)$ for the map $f_1$. We would like to define relative dynamical degrees of $f$ to be those of $f_1$. To this end, it suffices to show that if $f_2:X_2\rightarrow X_2$ is such another birational model of $f$ (i.e. $\pi _2:X_2\rightarrow Y$ is also a regular map), then $\lambda _p(f_1|\pi _1)=\lambda _p(f_2|\pi _2)$. This is the content of the next result. 

\begin{lemma}
Notations are as in the previous paragraph. Then for every $0\leq p\leq k-l$ we have
\begin{eqnarray*}
\lambda _p(f_1|\pi _1)=\lambda _p(f_2|\pi _2).
\end{eqnarray*}
\label{LemmaRelativeDynamicalDegreesGeneralCase}\end{lemma}   
\begin{proof}
Since $f_1$ and $f_2$ are two models of $f$, there is a birational map $\tau :X_2\rightarrow X_1$ such that $f_2=\tau ^{-1}\circ f_1\circ \tau$. Then using the arguments in the proof of Lemma \ref{LemmaBoundForDynamicalDegrees}, we can proceed as in the proof of Proposition 3.5 in \cite{dinh-nguyen} to complete the proof of the lemma. 
\end{proof}
 
\section{Dynamical degrees of semi-conjugate maps}
In this section we use the results in the previous sections to prove the "product formula" relating dynamical degrees of semi-conjugate maps, that is Theorem \ref{TheoremProductFormula}. As before, we consider dominant rational maps $f:X\rightarrow X$, $g:Y\rightarrow Y$ and $\pi :X\rightarrow Y$ for which $\pi \circ f=g\circ \pi$. Here $X$ and $Y$ are smooth projective varieties over an algebraic closed field of characteristic zero, of dimensions $k$ and $l$ correspondingly. 

We will adapt the arguments in Section 4 in \cite{dinh-nguyen}. To this end, it suffices that we have a refined Chow's moving lemma, adapted to a product of varieties $X\times Y$ (compare with  Propositions 2.3 and 2.4 in \cite{dinh-nguyen}). The following lemma is sufficient for estimating the degrees of strict transforms and strict intersections of varieties.  
\begin{lemma}
Let $X$ and $Y$ be smooth projective varieties over an algebraic closed field of characteristic zero. Let $\omega _X$ and $\omega _Y$ be generic hyperplane sections on $X$ and $Y$. Given $V$ a subvariety of $X\times Y$. Then, there exists a constant $A>0$ such that any subvariety $W\subset X\times Y$  can be written as 
\begin{eqnarray*}
W=W_{+}-W_{-},
\end{eqnarray*}
where $W_+$ is a member of a pencil of codimension $p$ varieties $\mathcal{W}_+$ whose generic members intersect $V$ properly, and $W_{-}$ is a subvariety of $X\times Y$.  Moreover, in $N^*(X\times Y)$
\begin{eqnarray*}
W_{+}\leq A\sum _{0\leq j\leq p, ~0\leq p-j\leq k-l}\alpha _j(W)\omega _X^j.\omega _Y^{p-j}. 
\end{eqnarray*} 
Here $\alpha _j(W)=W.\omega _X^{k-p+j}.\omega _Y^{l-j}$.
\label{LemmaChowMovingForProduct}\end{lemma}
 \begin{proof}
We will follow the ideas in Section 2 in \cite{dinh-nguyen-truong1}. Let us define $Z=X\times Y$. Let $X,Y$ be embedded into the same projective space $\mathbb{P}^N$. Let $\Delta _Z\subset Z\times Z$ be the diagonal, and let $\pi _1,\pi _2:Z\times Z\rightarrow Z$ be the projections. For simplicity, let us use  $\omega _X$ and $\omega _Y$ to denote also the pullbacks,  via the projections $Z=X\times Y\rightarrow X,Y$, of $\omega _X$ and $\omega _Y$ to $Z$. 

We note that $\pi _2^{-1}(W)$ intersects properly both $\pi _1^{-1}(V)$ and $\Delta _Z$ in $Z\times Z$, and moreover  
\begin{eqnarray*}
W=(\pi _1)_*(\pi _2^{-1}(W)\cap \Delta _{Z}).
\end{eqnarray*}
Hence, the proof of the lemma is finished if we can show that $\Delta _Z=\Delta _{+}-\Delta _{-}$ with the following three properties 

i) $\Delta _{+}$ is a member of a pencil  $\mathcal{D}_{+}$ of subvarieties of $Z\times Z$, and $\Delta _{-}$ is a subvariety of $Z$ intersecting properly with $\pi _2^{-1}(W)$.

ii) All members of $\mathcal{D}_+$ intersect $\pi _2^{-1}(W)$ properly. A generic member of $\mathcal{D}_{+}$ intersects $\pi _1^{-1}(V)\cap \pi _2^{-1}(W)$ properly.

iii) In $N^*(Z\times Z)$ we have
\begin{eqnarray*}
\Delta _{+}\leq A\sum _j\pi _1^*(\omega _X+\omega _Y)^j\wedge \pi _2^*(\omega _X+\omega _Y)^{k+l-j},
\end{eqnarray*}
for some positive constant $A$. (We note that when we expand the right hand side of the above inequality in terms of monomials in $\pi _{1}^{-1}(\omega _{X})$, $\pi _{2}^{-1}(\omega _{X})$,  $\pi _{1}^{-1}(\omega _{Y})$ and $\pi _{2}^{-1}(\omega _{Y})$, several terms will vanish since $\omega _X^{k+1}=0$ and $\omega _Y^{l+1}=0$.)

In fact, assume that we have the properties i), ii) and iii). Then we just need to choose $\mathcal{W}_+=(\pi _1)_*(\pi _2^{-1}(W)\cap \mathcal{D}_+)$ and $W_{-}=(\pi _1)_*(\pi _2^{-1}(W)\cap \Delta _{-})$.

Now we proceed to constructing $\Delta _{\pm}$ satisfying the properties i)-iii). We imbed $Z\times Z$ into $(\mathbb{P}^N)^4$ by using the embeddings $X,Y\subset \mathbb{P}^N$. By Chow's moving lemma, we can write 
\begin{eqnarray*}
\Delta _Z=\Delta _1-\Delta _2\pm \Delta _0,
\end{eqnarray*}
where $\Delta _1$ and $\Delta _2$ are the intersections of $Z\times Z$ with certain subvarieties of $(\mathbb{P}^N)^4$, and $\Delta _0$ is a subvariety of $Z\times Z$ intersecting properly both $\pi _2^{-1}(W)$ and $\pi _1^{-1}(V)\cap \pi _2^{-1}(W)$.  

If $\Delta _Z=\Delta _{1}-\Delta _{2}-\Delta _0$, then we let $\Delta _+=\Delta _1$ and $\Delta _{-}=\Delta _2+\Delta _0$. 

If $\Delta _Z=\Delta _{1}-\Delta _{2}+\Delta _0$ then we proceed as follows. We choose $L\subset (\mathbb{P}^N)^4$ a generic linear subspace of appropriate dimension. Let $C_L(\Delta _0)\subset (\mathbb{P}^N)^4$ be the cone over $Z$ with the vertex $L$. Then $C_L(\Delta _0)$ intersects $Z\times Z$ properly, and $\Delta _0$ is one component of $C_L(\Delta _0)\cap (Z\times Z)$. Then we define $\Delta _+=\Delta _1+C_L(\Delta _0)\cap (Z\times Z)$ and $\Delta _{-}=\Delta _2+[C_L(\Delta _0)\cap (Z\times Z)-\Delta _0]$. 

In any case, $\Delta _+$ is the intersection of $Z\times Z$ with a subvariety of $(\mathbb{P}^N)^4$ (note that in general this property is not true for $\Delta _{-}$). Hence, we can put $\Delta _+$ into a pencil $\mathcal{D}_+$. Here, $\mathcal{D}_+$ is the intersection of $Z\times Z$ with a pencil on $(\mathcal{P}^N)^4$. If we choose that pencil generic enough, we see that two properties i) and ii) that we required in the above are satisfied.  We note that here all members of $\mathcal{D}_+$ intersect $\pi _2^{-1}(W)$ properly since $\Delta _Z$ intersects $\pi _2^{-1}(W)$ properly (c.f. the construction used in the proof of Chow's moving lemma). Similarly, $\Delta _{-}$ intersects properly $\pi _2^{-1}(W)$. The last property iii) is also satisfied since the Chow's ring of $(\mathbb{P}^N)^4$ satisfies the Kunneth's formula (see Example 8.3.7 in \cite{fulton}).   
 \end{proof}
 
 \section{Surfaces and threefolds over a field of positive characteristic}

From the proofs of our results in the previous sections, it follows that whenever resolutions of singularities are available (or more generally if appropriate birational models of the maps under consideration are available), the results in the previous sections extend. In particular, since resolutions of singularities for surfaces (see \cite{zariski1, abhyankar1}) and threefolds (see \cite{zariski2, abhyankar2, cutkosky, cossart-piltant1, cossart-piltant2}) over a field of positive characteristic are available, this observation applies. We will consider these cases in more details in the below.

Let $X,Y$ be projective varieties over a field of positive characteristic. Let $f:X\rightarrow X$, $g:Y\rightarrow Y$ and $\pi :X\rightarrow Y$ be dominant rational maps such that $\pi \circ f=g\circ \pi$.  

\subsection{The case of surfaces}

We consider the case where $X$ is a surface. 

For relative dynamical degrees, the only non-trivial case is when $Y$ is a curve. In this case, there are only two relative dynamical degrees $\lambda _0(f|\pi )=1$ and $\lambda _1(f|\pi )$. Hence, the log-concavity for relative dynamical degrees becomes vacuous. The "product formula" becomes
\begin{eqnarray*}
\lambda _1(f)&=&\max \{\lambda _1(g),\lambda _1(f|\pi )\},\\
\lambda _2(f)&=&\lambda _1(g)\lambda _1(f|\pi ).
\end{eqnarray*}
As a consequence, if $f$ has an invariant fibration over a curve $Y$, then $\lambda _2(f)\geq \lambda _1(f)$. In particular, if $f$ is birational (i.e. $\lambda _2(f)=1$) with positive entropy(i.e.  $\lambda _1(f)>1$), then it is automatically primitive. Primitive birational maps of surfaces with zero entropy (that is $\lambda _1(f)=1$) have been recently classified in \cite{oguiso}, using results from \cite{diller-favre}. 

\subsection{The case of threefolds}

We consider the case where $X$ is a threefold. For relative dynamical degrees, the only non-trivial case is when $Y$ has dimension $1$ or $2$. These two cases are similar. For the case $\dim (Y)=1$, the "product formula" becomes
\begin{eqnarray*}
\lambda _1(f )&=&\max\{\lambda _1(g), \lambda _1(f|\pi )\},\\
\lambda _2(f )&=&\max \{\lambda _1(g)\lambda _1(f|\pi ), \lambda _2(f|\pi )\},\\
\lambda _3(f )&=&\lambda _1(g)\lambda _2(f|\pi ). 
\end{eqnarray*}
Therefore, if $f$ has an invariant fibration over a curve, then $\lambda _2(f)\geq \lambda _1(f)$ and $\lambda _1(f)\lambda _3(f)\geq \lambda _2(f)$.


\begin{thebibliography}{xx} 

\bibitem{abhyankar1}{S. Abhyankar,} \textit{Local uniformization on algebraic surfaces over ground fields of characteristic $p\not= 0$,} Annals Math., 2nd series, 63 (3), 491--526.

\bibitem{abhyankar2}{S. Abhyankar,} \textit{Resolutions of singularities of embedded algebraic surfaces,} 2nd edition, Acad. Press, 1998.

\bibitem{blanc-cantat}{J. Blanc and Cantat,} \textit{Dynamical degrees of birational transformations of projective surfaces,} arXiv: 1307.0361.

\bibitem{cossart-piltant1}{V. Cossart and O. Piltant,} \textit{Resolutions of singularities of threefolds in positive characteristic. I. Reduction to uniformization on Artin-Schreier and purely inseparable coverings,} J. Algebra 320 (3), 1051--1082.

\bibitem{cossart-piltant2}{V. Cossart and O. Piltant,} \textit{Resolutions of singularities of threefolds in positive characteristic. II.,} J. Algebra 321 (7), 1836--1976.

\bibitem{cutkosky}{S. D. Cutkosky,} \textit{Resolution of singularities for $3$-folds in positive characteristic,} Amer. J. Math. 131 (1), 59--127.

\bibitem{diller-favre}{J. Diller and C. Favre,} \textit{Dynamics of bimeromorphic maps of surfaces,} Amer. J. Math. 123 (2001), no. 6, 1135--1169.

\bibitem{dinh-nguyen}{T.-C. Dinh and V.-A. Nguyen,} \textit{Comparison of dynamical degrees for semi-conjugate meromorphic maps,} Comment. Math. Helv. 
{\bf 86} (2011), no. 4, 817--840.

\bibitem{dinh-nguyen-truong1}{T.-C. Dinh, V.-A. Nguyen and T. T. Truong,} \textit{On the dynamical degrees of meromorphic maps preserving a fibration,} Commun. Contemp. Math., 14, 1250042, 2012. arXiv: 1108.4792.

\bibitem{dinh-sibony1}{T-C Dinh and N. Sibony,} \textit{Regularization of currents and entropy,} Ann. Sci. Ecole Norm. Sup. (4), 37 (2004), no 6, 959--971.

\bibitem{dinh-sibony10}{T-C Dinh and N. Sibony,} \textit{Une borne sup\'erieure de l'entropie topologique d'une application rationnelle,} Annals of Math., 161 (2005), 1637--1644. 

\bibitem{dinh-sibony3}{T-C Dinh and N. Sibony,} \textit{Pullback of currents by holomorphic maps,} Manuscripta Math. 123 (2007), no. 3, 357--371.

\bibitem{esnault-oguiso-yu}{H. Esnault, K. Oguiso  and X. Yu,} \textit{Automorphisms of elliptic $K3$ surfaces and Salem numbers of maximal degree,} arXiv: 1411.0769.

\bibitem{esnault-srinivas}{H. Esnault and V. Srinivas,} \textit{Algebraic versus topological entropy for surfaces over finite fields,} arXiv: 1105.2426.

\bibitem{fornaess-sibony1}{J. E. Fornaess and N. Sibony,} \textit{Complex dynamics in higher dimensions,} Notes partially written by Estela A. Gavosto. NATO Adv. Sci. Inst. Ser. C Math. Phys. Sci. 459, Complex potential theory (Montreal, PQ, 1993), 131--186, Kluwer Acad. Publ., Dordrecht, 1994. 


\bibitem{friedlander-lawson}{E. M. Friedlander and H. B. Lawson,} \textit{Moving algebraic cycles of bounded degree,} Invent. math. 132 (1998), 91--119. 

\bibitem{fulton}{W. Fulton,} \textit{Intersection theory,} 2nd edition, Springer-Verlag Berlin Heidelberg, 1998.

\bibitem{griffiths-harris}{P. Griffiths and J. Harris,} \textit{Principles of algebraic geometry,} 1978, John Wiley and Sons, Inc.

\bibitem{gromov}{M. Gromov,} \textit{On the entropy of holomorphic maps,} Enseign. Math. (2), {\bf 49} (2003) no 3--4, 217--235. Manuscript (1977).

\bibitem{grothendieck}{A. Grothendieck,} \textit{Sur une note de Mattuck-Tate,} J. reine angew Math. 20 (1958), 208--215.

\bibitem{guedj4}{V. Guedj,} \textit{Ergodic properties of rational mappings with large topological degrees,} Annals of Math. 161 (2005), 1589--1607.


\bibitem{harris}{J. Harris,} \textit{Algebraic geometry: a first course,} Springer-Verlag New York, 1992.

\bibitem{hironaka}{H. Hironaka,} \textit{Flattening of analytic maps,} Manifolds-Tokyo 1973 (Proc. International Conf., Tokyo 1973), Univ. Tokyo Press, 1975, pp. 313--321.

\bibitem{ishii-milman}{S. Ishii and P. Milman,} \textit{The geometric minimal models of analytic spaces,} Math. Ann. 323 (2002), no 3, 437--451.

\bibitem{kollar}{J. Koll\'ar,} \textit{Lectures on resolutions of singularities,} Annals of mathematics studies, Princeton University press, 2007.

\bibitem{moishezon}{B. Moishezon,} \textit{Modifications of complex varieties and the Chow lemma,} Lecture Notes in Mathematics, no. 412, Classification of algebraic varieties and compact complex manifolds, Springer-Verlag Heidelberg 1974, pp. 133--139.

\bibitem{oguiso}{K. Oguiso,} \textit{Simple abelian varieties and primitive automorphisms of null entropy of surfaces,} arXiv: 1412.2535.

\bibitem{roberts}{J. Roberts,} \textit{Chow's moving lemma,} in Algebraic geometry, Oslo 1970, F. Oort (ed.), Wolters-Noordhoff Publ. Groningnen (1972), 89--96.


\bibitem{russakovskii-shiffman}{A. Russakovskii and B. Shiffman,} \textit{Value distributions for sequences of rational mappings and complex dynamics,} Indiana Univ. Math. J. 46 (1997), 897--932.

\bibitem{shafarevich}{I. R. Shafarevich,} \textit{Basic algebraic geometry 1,} 2nd revised and expanded version, Springer-Verlag Berlin Heidelberg New York 1994.

\bibitem{truong2}{T. T. Truong,} \textit{The simplicity of the first dynamical degree of a meromorphic map,} Accepted in Michigan Math. Jour.

\bibitem{xie}{J. Xie,} \textit{Periodic points of birational maps on projective surfaces,} arXiv: 1106.1825.

\bibitem{yomdin}{Y. Yomdin,}\textit{Volume growth and entropy,} Israel J. Math. {\bf 57} (3), (1987) 285--300.

\bibitem{zariski1}{O. Zariski,} \textit{The reduction of the singularities of an algebraic surface,} Annals Math. (2) 40 (3), 639--689.

\bibitem{zariski2}{O. Zariski,} \textit{Reduction of the singularities of algebraic three dimensional varieties,} Annals Math. (2) 45 (3), 472--542.

\end{thebibliography}
\end{document}